\documentclass[reqno,a4paper]{amsart}

\newif\ifproofs \newif\ifnotes \newif\ifrefchecks
\proofstrue
\notestrue
\refchecksfalse

\usepackage{url,colonequals}
\usepackage{amsmath,amssymb,amsmath,verbatim,comment,amsthm,xcomment,enumerate,hyperref}
\usepackage{stmaryrd}
\usepackage{tikz}
\usetikzlibrary{cd}

\numberwithin{equation}{section}
\usepackage[T1]{fontenc}
\usepackage[utf8]{inputenc}
\usepackage{lmodern}
\usepackage[english]{babel}
\theoremstyle{plain}

\makeatletter
\newcommand\@dotsep{4.5}
\def\@tocline#1#2#3#4#5#6#7{\relax
	\ifnum #1>\c@tocdepth 
	\else
	\par \addpenalty\@secpenalty\addvspace{#2}%
	\begingroup \hyphenpenalty\@M
	\@ifempty{#4}{%
		\@tempdima\csname r@tocindent\number#1\endcsname\relax
	}{%
		\@tempdima#4\relax
	}%
	\parindent\z@ \leftskip#3\relax \advance\leftskip\@tempdima\relax
	\rightskip\@pnumwidth plus1em \parfillskip-\@pnumwidth
	#5\leavevmode\hskip-\@tempdima{#6}\nobreak
	\leaders\hbox{$\m@th\mkern \@dotsep mu\hbox{.}\mkern \@dotsep mu$}\hfill
	\nobreak
	\hbox to\@pnumwidth{\@tocpagenum{#7}}\par%
	\nobreak
	\endgroup
	\fi}
\AtBeginDocument{%
\expandafter\renewcommand\csname r@tocindent0\endcsname{0pt}
}
\def\l@subsection{\@tocline{4}{0pt}{2.5pc}{5pc}{}}
\makeatother

\usepackage[color=orange]{todonotes}

\newtheorem{theorem}{Theorem}[section]
\newtheorem{corollary}[theorem]{Corollary}

\newtheorem{lemma}[theorem]{Lemma}
\newtheorem{prop}[theorem]{Proposition}
\theoremstyle{remark}
\newtheorem{remark}[theorem]{\rm\bf Remark}
\newtheorem{remarks}[theorem]{\rm\bf Remarks}
\newtheorem{definition}[theorem]{\rm\bf Definition}

\newtheorem{example}[theorem]{\rm\bf Example}

\def\mbb{\mathbb}

\def\mf{\mathfrak}

\def\mc{\mathcal}

\newcommand{\al}{\alpha}
\newcommand{\be}{\beta}
\newcommand{\ga}{\gamma}

\newcommand{\de}{\delta}

\newcommand{\om}{\omega}
\newcommand{\ph}{\varphi}

\newcommand{\si}{\sigma}
\renewcommand{\th}{\theta}

\newcommand{\Ga}{\Gamma}

\newcommand{\Om}{\Omega}

\newcommand{\Up}{\Upsilon}

\def\Proj{\mc{P}}
\def\g{\mf{g}}

\def\p{\mf{p}}
\def\q{\mf{q}}
\def\R{\mbb{R}}

\def\Z{\mbb{Z}}
\def\P{\mbb{P}}
\def\V{\mbb{V}}
\def\K{\mbb{K}}
\def\G{\mc{G}}
\def\D{\mc{D}}
\def\k{\mc{K}}
\def\N{\mc{N}}

\let\op=\operatorname

\def\d{\op{d}\!}
\def\Ad{\op{Ad}}

\usepackage[margin=2.3cm]{geometry}

\def\pmat#1{\begin{pmatrix}#1\end{pmatrix}}
\def\bmat#1{\begin{bmatrix}#1\end{bmatrix}}
\def\spmat#1{\left(\begin{smallmatrix}#1\end{smallmatrix}\right)}

\def\parder#1#2{\frac{\del#1}{\del#2}}

\let\del=\partial
\let\[=\llbracket
\let\]=\rrbracket
\let\<=\langle
\let\>=\rangle
\def\.{\hbox to5pt{\hss$\cdot$\hss}}
\let\un=\underline
\let\ov=\overline
\let\wh=\widehat

\let\wt=\widetilde
\let\o=\circ

\let\hr=\medskip

\newcount\notecnt
\def\NOTE#1{\bgroup \color{magenta} \texttt{#1} \egroup}
\newenvironment{note}{\bgroup \par \hr \color{magenta} \ttfamily}{\rmfamily \hr \egroup}

\usepackage[all]{xy}

\usepackage{environ}
\NewEnviron{killcontents}{}

\ifproofs
\else

\fi

\ifnotes\else
\presetkeys{todonotes}{disable}{}
\let\NOTE=\relax

\fi

\ifrefchecks\usepackage{refcheck}\fi

\begin{document}
\title{Canonical curves and Kropina metrics \\ in Lagrangian contact geometry}
\author{ Tianyu Ma, Keegan J. Flood, Vladimir S. Matveev, Vojt\v{e}ch \v{Z}\'{a}dn\'{i}k}
\address{K. J. F.:Faculty of Mathematics and Computer Science\\
  UniDistance Suisse\\
  Schinerstrasse 18\\
  Brig 3900\\
  Switzerland}
  \email{keegan.flood@unidistance.ch}
  \address{T. M.:Faculty of Mathematics\\
 National Research University Higher School of Economics, Moscow\\
  119048 Moscow\\ 
  Russia\\}
  \email{tianyuzero.ma@mail.utoronto.ca}
\address{V. S. M.:Institut f\"{u}r Mathematik\\
  Friedrich-Schiller Universit\"{a}t Jena\\
  07737 Jena\\ 
  Germany\\}
  \email{vladimir.matveev@uni-jena.de}
\address{V. \v{Z}.:Department of Mathematics and Statistics\\
  Masaryk University\\
  Kotl\'{a}\v{r}sk\'{a} 2\\
  Brno 61137\\
  Czech Republic}
  \email{zadnik@math.muni.cz}

\begin{abstract} 
We present a Fefferman-type construction from Lagrangian contact to conformal structures and examine several related topics.
In particular, we describe the canonical curves and their correspondence.
We show that chains and null-chains of an integrable Lagrangian contact structure are the projections of null-geodesics of the Fefferman space.
Employing the Fermat principle, we realize chains as geodesics of Kropina (pseudo-Finsler) metrics.
Using recent rigidity results, we show that ``sufficiently many'' chains determine the Lagrangian contact structure. 
Separately, we comment on Lagrangian contact structures induced by projective structures and the special case of dimension three.
\end{abstract}


\thanks{K. J. Flood and V. \v{Z}\'{a}dn\'{i}k gratefully acknowledge support from the Czech Science Foundation (GA\v{C}R) Grant 20-11473S 
and the grant 8J20DE004 of Ministry of Education, Youth and Sports of the Czech Republic.
T. Ma and V. Matveev were supported by DAAD (PPP 57509027) and DFG (MA 2565/6), and benefited from the hospitality from the Masaryk University. 
We also thank Sean N. Curry, Howard Jacobowitz, Josef \v{S}ilhan and Arman Taghavi-Chabert for useful discussions.
}

\date{\today} 

\maketitle

\tableofcontents


\section{Introduction}

An LC  (Lagrangian or Legendrian contact) structure on a smooth manifold $M$ consists of a contact distribution $\D\subset TM$ equipped with a decomposition $\D=E\oplus F$ such that both $E$ and $F$ are maximal isotropic subdistributions with respect to the Levi form on $\D$.
LC structures are of considerable interest as they are closely related to numerous classical topics in differential geometry and geometric differential equations, from projective geometry to symmetries of PDEs, see e.g. \cite{Takeuchi1994}, \cite{Hill2009}, \cite{Doubrov2020}.
Equivalently, a LC structure can be defined by an almost para-complex structure on $\D$ that is compatible with the Levi form.
The integrability of a LC structure means the integrability of the subdistributions $E$ and $F$ in the Frobenius sense, equivalently, the integrability of the almost para-complex structure in the Nijenhuis sense.
The latter definition is analogous to that of almost CR structures (more precisely, non-degenerate almost CR structures of hypersurface type).
In fact, both LC and CR structures can be seen as different real forms of a common complex structure.
CR geometry has seen continued development for decades, while LC geometry is comparatively unexplored.

In the present article, we focus on several topics related to the Fefferman construction, canonical curves and their correspondence.
We combine the abstract language of Cartan geometry with concrete coordinate expressions, algebraic model observations with analytic techniques, etc.
We now summarize what can be found in individual sections.
 
We start with a careful description of the homogeneous model. Besides the standard interpretation of the model LC structure as a flag variety of particular type, we develop another interpretation as a 
para-complex projectivization of the null-cone of para-Hermitian space, see Section \ref{sec-LC-model}.
This provides both the closest analogue to the model description of CR structures and the key instrument for many later accounts.
An affine realization of the previous picture and its potential deformations yield the notion of induced LC structure on a generic hypersurface in a para-complex space.
It is shown in Proposition \ref{prop-embed} that a LC structure can be locally realized this way if and only if it is integrable.
This contrasts the CR situation, where the problem is much more intricate and still partially open, see the discussion at the beginning of Section \ref{localembedability}. 

Concerning the classical Fefferman construction, it yields a circle bundle over a CR manifold which is equipped with a conformal class of metrics.
It allows several descriptions and has many applications, cf. \cite{Fefferman1976}, \cite{Lee1986}, \cite{Farris1986}, \cite{Cap2005}.
Our adaptation of the construction for LC structures in Section \ref{sec:Fefferman construction} is based on the general scheme in the framework of parabolic geometries as in \cite{Hammerl2017}.
The rough portrayal generalizes the model observations to the curved setting, while an explicit description of a representative metric from the conformal class is a more subtle task.
For integrable LC structures, we accomplish this by suitably calibrating a Cartan gauge, expressed in adapted local coordinates, see Theorem \ref{prop-Feff-integrable}.
The resulting formula is quite direct in the sense that it involves only the defining functions (and their partial derivatives) of the LC structure.
In particular, we do not use any compatible affine connection of Tanaka--Webster type, which are popular in the CR literature.

The most prominent canonical curves for LC structures, just as in CR geometry, are the chains.
On the one hand, they exhibit geodesic-like properties in the sense that every unparametrized chain is uniquely determined by an initial direction, provided that it is transverse to the contact distribution, cf. \cite{Cartan1933}, \cite{Chern1974}, \cite{Jacobowitz1985}.
On the other hand, chains form a more complicated system of curves as they cannot be geodesics of any affine connection.
Chains also play an important role in the rigidity of structures: 
both for CR and LC structures, the path geometry of chains determines the structure so that chain preserving diffeomorphisms are either isomorphisms or anti-isomorphisms of the structure, see \cite{Cheng1988}, \cite{vcap2009}.
Following the results of \cite{Cheng2019}, we generalize the previous conclusions for integrable LC structures so that they are still valid even if the whole family of chains shrinks to a ``sufficiently big'' subset.
Details are specified in Section \ref{sec-Kropina}, where the whole discussion culminates.
The key objects there are the Kropina metrics, which are metrics of pseudo-Finsler type, defined (off the contact distribution) via a representative metric on the Fefferman space.
The point is that chains are precisely the geodesics of any such constructed Kropina metric, see Theorem \ref{prop-Kropina}.
Besides the just mentioned applications, this realization provides an efficient tool for deriving ODE systems for chains, namely, as Euler--Lagrange equations of the corresponding functional.

The preceding discussion is based on a careful analysis of the correspondence of curves under the Fefferman projection.
The outcomes are packed in Theorem \ref{cor:integrable chains}, where chains of an integrable LC manifold are identified with projections of null-geodesics of the Fefferman conformal structure that are not perpendicular to the vertical subbundle of the projection.
Projections of null-geodesics that are perpendicular to the vertical subbundle are identified with null-chains, the contact canonical curves which also have a counterpart in CR geometry, cf. \cite{Koch}.
It is worth mentioning that, in the present article, this correspondence is the content of theorem rather than definition, which is often the case in CR references.
All canonical curves encountered in this article are defined uniformly in the framework of parabolic geometries and specified by a subset of the Lie algebra which underlies the structure in question.
This primarily determines canonical curves in the homogeneous model, the passage to the general curved setting is provided by the notion of the development of curves, see Section \ref{curves-prelim}.
This is why we devote some space to model interpretations both of chains and null-chains in Section \ref{curves-model}.
Besides their auxiliary purposes, they are also of interest on their own, cf. \cite{Bor2022}.

An important class of LC structures consists of those induced by projective structures, see \cite{Takeuchi1994}.
In higher dimensions, such structures are only half-integrable (unless flat) and have no analogy in the CR case.
We comment on some of the previously discussed topics for this class in Section \ref{LC-by-projective}.
On the way, we obtain a characterization of which LC structures come from projective structures, formulated in terms of defining functions, see Proposition \ref{LC-proj-char}.
This result generalizes Cartan's criterion in the 3-dimensional case, which we recall in Section \ref{sec-dim3}. 
In this dimension, every LC structure is automatically integrable and equivalent to a path structure on a 2-dimensional leaf space.
In this case, we can confront our general formulas with their elusive explicit companions, cf. \cite{Nurowski2003}, \cite{Bor2022}.
For the sake of illustration, we also elaborate on some ideas in more detail, see our approach to Proposition \ref{prop: Bor-Wilse}, originally proved in \cite{Bor2022}, and Example \ref{ex}.


\section{Cartan geometries and LC structures}

In this section, we collect the background for LC structures that is used throughout the article. 
LC structures can be described in several equivalent ways, which are presented in Section \ref{sec-LC-general}.
One of the approaches is in terms of (parabolic) Cartan geometries, and it is this approach which we predominantly employ.
Section \ref{sec-Cartan} is devoted to treating the necessary background concerning Cartan geometry.
Section \ref{sec-LC-general} then treats LC geometry.
There is nothing truly novel in these two subsections, basic references for this part are  \cite{Cap2009}, \cite{Sharpe1997}.
In Section \ref{sec-LC-model}, we discuss in detail several interpretations of homogeneous models for LC structures. 
Besides the common ones, we discuss a para-complex analogue of the standard model for CR structures (the CR quadrics). 
This interpretation is by no means surprising but, to our knowledge, nowhere published.
The development in later sections relies heavily on these observations.

\subsection{Cartan and parabolic geometries} \label{sec-Cartan}
Given a Lie group $G$ and a closed subgroup $P\subset G$, let $\p\subset\g$ be the corresponding pair of Lie algebras.
The homogeneous model for the Cartan geometry of type $(G,P)$ consists of the homogeneous space $G/P$, the $P$-principal bundle $G\to G/P$ and the Maurer--Cartan form $\om:TG\to\g$.
This data encodes the geometric structure on $G/P$ in Klein's sense.
In addition, the Maurer--Cartan form defines the absolute parallelism on $G$, reproduces the infinitesimal generators of the principal $P$-action on $G$, and is $P$-equivariant.

Abstracting this picture leads to the notion of general Cartan geometry of type $(G,P)$:
it consists of a base manifold $M$, a $P$-principal bundle $\G\to M$ and a \emph{Cartan connection}, which is a $\g$-valued 1-form $\om:T\G\to\g$ satisfying the following three properties,
\begin{itemize}
\item $\om_z:T_z\G\to\g$ is a linear isomorphism, for each $z\in\G$,
\item $\om\left(\frac{d}{dt}\big\vert_{0}\, r_{\exp(tX)}(z)\right) =X$, for each $X\in\p$ and $z\in\G$,
\item $(r_p)^*\om =\Ad_{p^{-1}}\o\,\om$, for each $p\in P$,
\end{itemize}
where $r_p:\G\to\G$ and $\Ad_p:\g\to\g$ denote the right multiplication action and the adjoint action, respectively, of an element $p\in P$.
The composition of $\om:T\G\to\g$ with the quotient projection $\g\to\g/\p$ provides an identification of the tangent bundle $TM$ with the associated bundle with the standard fiber $\g/\p$
(whose $P$-module structure is induced by the adjoint action),
\begin{align}
TM\cong\G\times_{P} \g/\p .
\label{eq-TM}
\end{align}
A morphism of two Cartan geometries of the same type is a morphism of the corresponding principal bundles that preserves the Cartan connections.

The \emph{curvature} of the Cartan connection $\om$ is the $\g$-valued 2-form on $\G$ defined by 
\begin{align}
\Om (u,v):= \d\omega (u,v) +[\omega(u),\omega(v)],
\label{eq-curvature}
\end{align}
for vector fields $u$ and $v$ on $\G$, where $[\ ,\ ]$ denotes the Lie bracket in $\g$.
The curvature vanishes identically if and only if the Cartan geometry is locally isomorphic to the homogeneous model.
The curvature $\Om$ is strictly horizontal, i.e. it vanishes under insertion of any vertical vector.
This means that the curvature is represented by a $P$-equivariant map $\G\to\wedge^2(\g/\p)^*\otimes\g$, the so-called \emph{curvature function}.
Further composition with the projection $\g\to\g/\p$ yields a map $\G\to\wedge^2(\g/\p)^*\otimes(\g/\p)$ representing a tensor field of type $\wedge^2 T^*M \otimes TM$, the \textit{torsion} of the Cartan connection $\om$.

As in the model case, the Cartan-geometric view provides a new perspective and tools for studying the underlying geometric structure.
The relationship between the underlying structure and the corresponding Cartan-geometric data can often be made bijective by imposing certain conditions on the Cartan curvature.
In particular, for a large family of parabolic geometries, this is guaranteed by the notions of regularity and normality.
Most of the structures discussed below belong to this family and in all these cases the regularity condition is forced by the normality.
The key concepts can be introduced as follows.

A \emph{parabolic geometry} is a Cartan geometry of type $(G,P)$ where $G$ is a semisimple Lie group and $P\subset G$ is a parabolic subgroup.
Parabolic subalgebras $\p\subset\g$ are related to gradings of semisimple Lie algebras as follows.
Let
\begin{align}
\label{eqn:parabolic grading alg}
\g=\g_{-k}\oplus\dots\oplus\g_0\oplus\dots\oplus\g_k
\end{align}
be a grading of depth $k$ of a semisimple Lie algebra $\g$, 
i.e. $[\g_i,\g_j]\subseteq\g_{i+j}$, where $\g_l=0$, for $|l|>k$. 
If the subalgebra $\p_+:=\g_1\oplus\dots\oplus\g_k$ is generated by $\g_1$
or, equivalently, $\g_-:=\g_{-k}\oplus\dots\oplus\g_{-1}$ is generated by $\g_{-1}$,
then $\p:=\g_{0}\oplus\p_{+}$ is a parabolic subalgebra with nilradical $\p_+$. 
The obvious $P$-invariant filtration of $\g$ induces a filtration of the tangent bundle $TM$, yielding an associated graded module $\op{gr}(TM)$. 
The parabolic geometry is called \emph{regular} if the algebraic bracket on $\op{gr}(TM)$ induced by the Lie bracket in $\g$ agrees with the Levi bracket on $\op{gr}(TM)$ induced by the commutator of vector fields on $M$.

The Killing form on $\g$ provides an isomorphism of $P$-modules $(\g/\p)^*\cong\p_+$, hence the curvature function of a parabolic geometry is seen as a function $\G\to\wedge^2\p_+\otimes\g$.
The natural normalization condition is provided by the \emph{Kostant codifferential}, which is the codifferential in the complex 
\begin{equation*}
\begin{tikzcd}
  \cdots \ar[r,"\partial^*"] & \wedge^2\p_+\otimes\g \ar[r,"\partial^*"] & \p_+\otimes\g \ar[r,"\partial^*"] & \g \ar[r] & 0
\end{tikzcd}
\end{equation*}
determining the Lie algebra cohomology of $\p_+$ with coefficients in $\g$.
The parabolic geometry is called \emph{normal} if its curvature function has values in $\ker\del^*$.
In such cases, the composition of the curvature function with the quotient projection onto the cohomology space $\ker\del^*\to\ker\del^*/\op{im}\del^*$ yields the so-called \textit{harmonic curvature}.
For regular and normal parabolic geometries, this is a fundamental curvature object which determines the entire curvature.
In particular, its vanishing is equivalent to the local flatness of the parabolic geometry.
The harmonic curvature typically consists of only a few components that can be interpreted in underlying geometric terms. 

Instead of the Cartan connection, i.e. a globally defined 1-form $\om:T\G\to\g$, one may consider the \emph{Cartan gauges}, i.e. local pull-backs $\phi^*\om:TU\to\g$, where $U\subset M$ and $\phi:U\to\G$ is a section of the bundle projection $\G\to M$.
Under a change of section $\wh\phi=r_p\o\phi$, for $p:U\to P$, the corresponding gauge is expressed as
\begin{align}
\wh\phi^*\om = \Ad_{p^{-1}}(\phi^*\om) +p^*\om_P , 
\label{eq-gauge-change}
\end{align}
where $\om_P:TP\to\p$ denotes the Maurer--Cartan form on $P$.
In this way, the Cartan geometry is treated via an atlas of Cartan gauges with the equivalence relation given by \eqref{eq-gauge-change} on intersections of domains.
Choosing a basis of $\g$, any gauge is represented by a cluster of ordinary 1-forms on the base manifold which allow a very straightforward analysis.
The Cartan gauge approach is repeatedly used in the article.

\subsection{LC structures} \label{sec-LC-general}
Let $M$ be a smooth manifold of odd dimension with a contact distribution $\D\subset TM$.
This means that $\D$ has codimension one and it is maximally non-integrable, equivalently, the Levi form $\mc{L}:\wedge^2\D\to TM/\D$, given by the Lie bracket of vector fields as $\mc{L}(X,Y)=[X,Y]\!\mod\D$, is non-degenerate.

\begin{definition}
An \emph{LC structure} on a manifold $M$ of dimension $2n+1$ consists of a contact distribution $\D\subset TM$ equipped with a decomposition $\D=E\oplus F$ into transversal subdistributions of rank $n$ which are both isotropic with respect to the Levi form.
If either $E$ or $F$ is an integrable distribution then the LC structure is called \emph{half-integrable}, if both $E$ and $F$ are integrable then the structure is called \emph{integrable}.
\end{definition}

To get closer to the CR situation, a LC structure can be equivalently described by an almost para-complex structure $K$ on the distribution $\D$ such that the Levi form is of type $(1,1)$ with respect to $K$.
This means that $K$ is an endomorphism of $\D$ such that $K\o K=\op{id}$ and $\mathcal{L}(KX,KY)=-\mathcal{L}(X,Y)$, for all $X,Y\in\D$.
The $\pm 1$-eigenspaces of $K$ then define the transversal maximal isotropic subdistributions $E$ and $F$.
The integrability of the LC structure is equivalent to the vanishing of the Nijenhuis tensor of $K$.
The map $\Phi:\odot^2\D\to TM/\D$, defined by 
$\Phi(X,Y):=\mathcal{L}(KX,Y) =-\mathcal{L}(X,KY)$, is clearly symmetric and non-degenerate.
The null-vectors of $\Phi$ form a cone in $\D$ containing $E$ and $F$; they are called the \emph{null-vectors} of the LC structure.

A morphism of LC structures is a morphism of contact manifolds that preserves the Legendrian decompositions, or equivalently, the partial para-complex structures.

Locally, the contact distribution $\D\subset TM$ is the kernel of a contact form, i.e. a 1-form $\si$ such that the top-form $\si\wedge(\d\si)^n$ is nowhere vanishing.
There is a 1-parameter freedom in the choice of contact form, so they play a role of scales.
In this article, we only deal with (at least) half-integrable LC structures.
Let us keep the convention that the subdistribution $F\subset\D$ is always integrable.
By the Darboux and Frobenius Theorems,  the local coordinates $(x^i,u,p_j)$ on $M$ can always be chosen so that $\si=\d u -p_i \d x^i$ and the leaves of $F$ correspond to $(x^i)$ and $u$ constant.
Here and below, the indices $i,j$ etc., run from 1 to $n$ and repeated indices indicate the sum 
(i.e. the summation symbol is omitted if there is no danger of confusion).
In such coordinates, the LC structure is fully determined by a collection of functions $f_{ij}=f_{ij}(x^k,u,p_l)$ so that $f_{ij}=f_{ji}$ and 
\begin{align}
E = \left\< \parder{}{x^i} +p_i\parder{}{u} +f_{ij}\parder{}{p_j} \right\>, \quad
F = \left\< \parder{}{p_i} \right\> .
\label{eq-frame}
\end{align} 
Dually, the LC structure is described by the coframe
\begin{align}
\si=\d u -p_i \d x^i, \quad
\th^i=\d x^i, \quad
\pi_i=\d p_i -f_{ij}\d x^j ,
\label{eq-coframe}
\end{align}
so that $E=\ker\<\si,\pi_i\>$ and $F=\ker\<\si,\th^i\>$.
Note that the forms are related by  $\d\si=\th^i\wedge\pi_i$.
A coframe satisfying the previous three conditions is called \emph{adapted}.
There is a 1-parameter family of adapted coframes corresponding to the choice of contact form, namely, 
\begin{align}
\wh\si=e^{2f}\si, \quad
\wh\th^i=e^f(\th^i-2f^i\si), \quad
\wh\pi_i=e^f(\pi_i+2f_i\si),
\label{eq-coframe-change}
\end{align}
where $f$ is an arbitrary function on $M$ and the functions $f^i$ and $f_i$ are given by the (unique) decomposition $\d f=f^i\pi_i+f_i\th^i+f_0\si$.
In particular, there is a conformal class of symplectic forms $\d\si=\th^i\wedge\pi_i$, or split-signature metrics $\th^i\odot\pi_i$, on $\D$.

\begin{remarks}
To a half-integrable LC structure described by the defining functions $f_{ij}$ as in \eqref{eq-frame} or \eqref{eq-coframe}, one may associate the following system of 2nd-order PDEs
\begin{align}
\label{eqn: LC PDE}
\frac{\partial^2 u}{\partial{x^i}\partial{x^j}}= f_{ij} \left(x^k,u,\frac{\partial u}{\partial x^k} \right) ,
\end{align}
for the unknown $u=u(x^i)$.
It follows that the compatibility of this system is equivalent to the integrability of the LC structure, cf. \cite[Lemma 2.3]{Doubrov2020}.
In such case, the leaves of $E$ correspond to graphs of solutions in the $(x^i,u)$-space. 

In the 3-dimensional case, i.e. for $n=1$, the LC structure is automatically integrable and the previous system reduces to a single second-order ODE.
This encodes the so-called \emph{path geometry} in the underlying 2-dimensional space.
More details on this specific situation are located in Section \ref{sec-dim3}.
\end{remarks}

For our purposes, the key fact is that LC structures are the underlying structures of parabolic geometries corresponding to the  contact grading of simple Lie algebras of the $\mf{sl}$-type.
To be more precise, for $\g=\mf{sl}(n+2,\R)$,  the contact grading can be described by the following block-matrix decomposition
\begin{align}
\label{eq-LC-g}
\g=\begin{pmatrix}
\g_0 & \g_1^E & \g_2\\
\g_{-1}^E & \g_0 & \g_1^F \\
\g_{-2} & \g_{-1}^F & \g_0
\end{pmatrix} ,
\end{align}
where the blocks have sizes $1$, $n$ and $1$ along the diagonal.
As the group $G$ with the Lie algebra $\g$ we take $G=SL(n+2,\R)$.
The subgroup $P\subset G$ with the Lie algebra $\p\subset\g$ is the one consisting of block upper triangular matrices according to the previous schematic description.
The homogeneous space $G/P$ is studied in detail in Section \ref{sec-LC-model}.
The whole tangent space corresponds to the $P$-module $\g/\p$, cf. \eqref{eq-TM}, and has dimension $2n+1$.
Under the linear isomorphism $\g/\p\cong\g_-$, the $P$-action on $\g/\p$ is truncated to $\g_-$.
In this manner, the contact distribution corresponds to $\g_{-1}\subset\g_-$, the Legendrian decomposition to $\g_{-1}=\g_{-1}^E\oplus\g_{-1}^F$ and the Levi form to the bracket $\wedge^2\g_{-1}\to\g_{-2}$.
The contact null-vectors correspond to the elements 
$\spmat{ 0 & 0 & 0 \\ X & 0 & 0 \\ 0 & Y^t & 0 } \in \g_{-1}$ such that $Y^t X=0$.

Normal parabolic geometries of the current type are always regular.
The harmonic curvature decomposes as follows.
In the 3-dimensional case, i.e. for $n=1$, there are two components, both of homogeneity 4 which, in particular, means the geometry is torsion-free.
In the case of higher dimension, there are two components of homogeneity 1 and one component of homogeneity 2, which correspond to two torsions and one curvature of Weyl type, respectively.
Vanishing of either torsion component is equivalent to the integrability of the corresponding subdistribution.

The moral of the present preparation can be summarized as follows.

\begin{prop}[{\cite[Section 4.2.3]{Cap2009}}] 
For the pair $P\subset G$ as above, the category of normal parabolic geometries of type $(G,P)$ is equivalent to the category of LC structures.
Under this correspondence, (half-)torsion-free Cartan connections  correspond to (half-)integrable LC structures.
\end{prop}

\subsection{The homogeneous model in detail} \label{sec-LC-model}

The homogeneous model for the LC geometry is given by $G/P$, where $G=SL(n+2,\R)$ and $P\subset G$ is the block upper-triangular subgroup with blocks of sizes 1, $n$ and 1 as above.
The common interpretation of $G/P$ is given as follows.
We denote the standard and the dual bases of $\R^{n+2}$ and $\R^{(n+2)*}$ by $(e_0,e_i,e_{n+1})$ and $(e^0,e^i,e^{n+1})$, respectively, where  $i=1,\dots,n$.
The standard action of $G$ on $\R^{n+2}$ yields a transitive action of $G$ on any flag variety of subspaces in $\R^{n+2}$.
The subgroup $P$ stabilizes the flag $\<e_0\> \subset \<e_0,e_1,\dots,e_n\>$. Thus $G/P$ is identified with the flag variety of type $(1,n+1)$ in $\R^{n+2}$.

Another interpretation of the model is given by projectivizing. Then previous flag variety corresponds to the set of incident pairs of points and hyperplanes in real projective space $\R\P^{n+1}$.
Noting that hyperplanes are equivalent to kernels of 1-forms, up to a scalar multiple, further identifies $G/P$ with the projectivized cotangent bundle of real projective space $\R\P^{n+1}$.
In this context, the subgroup $P\subset G$ can be realized as $P=\un{P}\cap\ov{P}$, where $\un{P}$ is the stabilizer of $\<e_0\>$ and $\ov{P}$ is the stabilizer of $\<e_0,e_1,\dots,e_n\>$.
Since $\<e_0,e_1,\dots,e_n\>=\ker(e^{n+1})$, $\ov{P}$ is the stabilizer of $\<e^{n+1}\>$ under the dual action of $G$ on $\R^{(n+2)*}$.
This gives the standard double fibration:
\begin{align}
\begin{split}
\xymatrix@R=.7\baselineskip{
& G/P \ar[dl] \ar[dr] & \\
G/\un{P}\cong\R\P^{n+1} & & G/\ov{P}\cong\R\P^{(n+1)*} }
\end{split}
\label{dia-double-fiber}
\end{align}

The previous picture can also be drawn on the para-complex canvas so that the resulting interpretation most closely resembles the standard CR model.
Since para-complex numbers do not form a field, one has to be careful with some basic notions.
To start with, we introduce the necessary background, partly reacting to and extending that of \cite[Section 3.2]{Hammerl2017}.

By a \emph{para-complex vector space} of dimension $n+2$ we mean a real vector space of dimension $2n+4$ equipped with a para-complex structure $\K$, i.e. an endomorphism of $\V$ such that  $\K\o\K=\op{id}$ and the corresponding $\pm1$-eigenspaces, denoted by $\V_\pm$, have the same dimension.
The complement of these subspaces is denoted by $\V_0:=\V\setminus(\V_+\cup \V_-)$.
A \emph{para-complex subspace} of dimension $d$ is a real subspace of $\V$ of dimension $2d$ such that its intersection with both $\V_+$ and $\V_-$ has real dimension $d$.
For $d=1$ and 2, we speak about para-complex lines and planes, respectively. 
Note that any para-complex subspace is $\K$-invariant, but the opposite is not true in general.
For example, the para-complex hull $\<v,\K(v)\>$ of an element $v\in\V$ is a para-complex line if and only if $v\in\V_0$.

A \emph{para-Hermitian vector space} is a para-complex vector space equipped with an inner product, denoted by $\.$, with respect to which $\K$ is skew.
This compatibility means that the $\pm1$-eigenspaces $\V_\pm$ are both null and are dual one another.
In particular, the inner product must have the split signature.
A choice of basis of $\V_+$, together with the dual basis of $\V_-\cong\V_+^*$, provides an identification of $\V$ with $(n+2)$-ary Cartesian power of para-complex numbers with their standard norm; this is denoted by $\R^{n+2,n+2}$.

Let a para-complex line $L\subset\V$ be represented as $\<v_+,v_-\>$, where $v_\pm\in\V_\pm$.
Under the identification $\V_-\cong\V_+^*$, the intersection $L\cap\V_-=\<v_-\>$ determines a hyperplane $\ker v_-\subset \V_+$ which contains $v_+\subset\V_+$, if and only if $v_-\.v_+=0$.
Thus, the pair $(\<v_+\>,\<v_-\>)$ represents a flag in $\V_+\cong\R^{n+2}$ if and only if the para-complex line $\<v_+,v_-\>$ is null.
Altogether, the flag variety $G/P$ is identified with the set of para-complex null-lines in $\V$.
Forgetting about the para-complex structure $\K$, the previous is identified the set of real null-planes in $\V$ having a non-trivial intersection with both $\V_+$ and $\V_-$.

The principal group $G$ is naturally seen as a subgroup in the special orthogonal group of the pseudo-euclidean space $\V$.
This embedding, denoted by 
\begin{align}
\eta: G=SL(n+2,\R) \hookrightarrow SO(n+2,n+2)=\wt{G} ,
\label{eq-eta}
\end{align}
is given by the standard and the dual action of $G$ on $\V_+\cong\R^{n+2}$ and $\V_-\cong\R^{(n+2)*}$, respectively. 
In this way, we may think of $G$ as a special ``para-unitary group''.
Adapting the notation from above, the subgroup $P\subset G$ is seen as the stabilizer of the para-complex null-line $\<e_0,e^{n+1}\>=:O$.

The preceding discussion is summarized in the following proposition:

\begin{prop}
\label{prop-LC-model}
Let $M=G/P$ be the homogeneous model for a LC structure of dimension $2n+1$.
Then $M$ is naturally identified with each of the following:
\begin{enumerate}[(a)]
\item \label{LC-model-a}
The flag variety of type $(1,n+1)$ in the real vector space $\R^{n+2}$.
\item  \label{LC-model-b}
The set of incident pairs of points and hyperplanes in the real projective space $\R\P^{n+1}$.
\item  \label{LC-model-c}
The projectivized cotangent bundle of the real projective space $\R\P^{n+1}$.
\item \label{LC-model-d}
The set of real null-planes in the pseudo-euclidean space $\R^{n+2,n+2}$ having a non-trivial intersection with both $\R^{n+2}\times\{0\}$ and $\{0\}\times\R^{n+2}$.
\item \label{LC-model-e}
The set of para-complex null-lines in the para-Hermitian space $\R^{n+2,n+2}$.
\end{enumerate}
\end{prop}

\begin{remark}
\label{rmk:model 2.5}
Denoting the null-cone of non-zero null-vectors in $\V$ by $\N$, the canonical projection $\mathcal{N}\rightarrow G/P$ is surjective when restricted to the open subset $\mc N_0:=\mc N\setminus(\V_+\cup \V_-)$. 
In other words, the interpretation in \eqref{LC-model-e} can be described as the \emph{para-complex projectivization}  of $\N_0$.
Thus we realize $M=G/P$ as a hyperquadric in the \emph{para-complex projective space}, the space of para-complex lines in $\V$.
In this framework, the double fibration \eqref{dia-double-fiber} is completed by the real projectivizations of $\V_+$ and $\V_-$.

For an affine para-complex hyperplane in $\R^{n+2,n+2}$, its intersection with $\N_0$ provides another, but not global, interpretation of $M$ as 
a hyperquadric in para-complex space $\R^{n+1,n+1}$.
It follows that deforming this view, i.e. considering general hypersurfaces in $\R^{n+1,n+1}$, yields a (local) realization of any LC structure provided that it is integrable.
See Section \ref{localembedability} for details.
\end{remark}

For later use, we add more details on the tangent bundle of the model LC manifold.
On the one hand, $M$ is the homogeneous space $G/P$, hence the tangent space $T_LM$ is identified with $\g/\p$, for any $L\in M$, via the Maurer--Cartan form, cf.  \eqref{eq-TM}.
On the other hand, $M$ is a submanifold in the Grassmannian of vector subspaces (of real-dimension 2) in $\V$. 
It is well-known that the tangent space in $L$ of the Grassmannian is identified with $L^* \otimes \V/L$, hence we have $T_L M \subset L^* \otimes \V/L$, for any $L\in\N_0$.
These two perspectives will be freely combined later, so we should understand them better.
To specify the latter one, we need the natural projection $q:\V/L\to \V/L^\perp$, induced by the inclusion $L\subset L^\perp$, and the natural isomorphism $\V/L^\perp\cong L^*$:

\begin{lemma} \label{lem:TOM}
For any $L\in M$, the tangent space $T_LM \subset L^*\otimes\V/L$ is identified with the space of para-complex linear maps $w:L\to\V/L$ such that $q\o w:L\to L^*$ is skew.
Under this identification, the contact distribution $\D_L\subset T_LM$ corresponds to $q\o w=0$, i.e. it is isomorphic to the space of para-complex elements in $L^* \otimes L^\perp/L$.

Any $L\in M$ and $w\in T_LM$ determines a $\K$-invariant vector subspace $W:=\op{im}(w)+L\subset\V$, where $\op{im}(w)$ denotes the image of $w$, interpreted as an element of $L^*\otimes \V/L$.
The type of a non-zero vector $w$ can be read from the dimension and the restricted inner product on the corresponding subspace $W$ as follows:
\begin{enumerate}[(a)]
\item $w$ belongs to $E$ or $F$ if and only if $W$ has real dimension 3 and is null, 
\item $w$ is a null-vector from $\D\setminus(E\cup F)$ if and only if $W$ has real dimension 4 and is null,
\item $w$ is a non-null vector from $\D$ if and only if $W$ has real dimension 4 and is half-degenerate, 
\item $w$ is a transverse vector to $\D$ if and only if if $W$ has real dimension 4 and is non-degenerate.
\end{enumerate}
\end{lemma}

\begin{proof}
To describe the typical fiber of $TM$, we restrict to the origin $O=\<e_0,e^{n+1}\>$, the para-complex line in $\V$ that is fixed by the subgroup $P\subset G$. 
To specify the inclusion  $T_OM\cong\g/\p \subset O^*\otimes\V/O$, let us consider the map
\begin{align}
\g/\p \ni
\bmat{ \. & \. & \. \\ X^i & \. & \. \\ z & Y_j & \. }
\longmapsto
\bmat{ e_0 \mapsto X^i e_i + z e_{n+1} \\ e^{n+1} \mapsto Y_j e^j - z e^0 } 
\in O^*\otimes\V/O.
\label{eq-TOM}
\end{align}
This map is clearly injective, linear, para-complex and the composition $q\o w\in O^*\otimes O^*$ is skew.
Thus, it is an isomorphism onto the image, which is just the space described in the statement.
To conclude, one easily checks that the map is also $P$-equivariant, i.e. commutes with the $P$-action on $\g/\p$ and $O^*\otimes\V/O$, respectively.
The description of the contact subspace is clear:
the contact subspace corresponds to $\g_{-1}\subset\g$, which is given by $z=0$ in the description above, and its image under the map \eqref{eq-TOM} is $O^*\otimes O^\perp/O$.

The subspace $W\subset\V$ associated to tangent vector as in \eqref{eq-TOM} is $\< X^i e_i + z e_{n+1}, Y_j e^j - z e^0, e_0, e^{n+1} \>$. 
It is clearly $\K$-invariant and its real dimension is at least 2.
Generically, the dimension is 4. In special cases, it is 3, respectively 2.
The two special cases happens if and only if $z=X=0$ or $z=Y=0$, respectively $z=X=Y=0$, which corresponds to elements from $E$ or $F$, respectively to the zero vector.
To distinguish the types, let us take the Gram matrix of the inner product induced on $W$:
\begin{align*}
\pmat{ 0 & a & 0 & z \\ a & 0& -z & 0 \\ 0 & -z & 0 & 0 \\ z & 0 & 0 & 0 } ,
\end{align*}
where $a=X^iY_i$.
The cases (a--b) correspond to $a=z=0$, 
the case (c) corresponds to $a\ne0$ and $z=0$, 
and the case (d) corresponds to $a\ne0$ and $z\ne0$. 
Hence the claim follows.
\end{proof}

Notice that, in all cases except (a), the subspace $W\subset\V$ is a para-complex plane. 
In particular, in the non-degenerate case (d), the signature of the inner product on $W$ is split.

\subsection{Local embeddability}\label{localembedability}
The problem of local embeddability concerns whether a given structure can be locally realized as the induced structure on a submanifold of an ambient space with a related structure. 
In the context of CR geometry, this means a realization as a hypersurface in complex space, with the contact distribution being the maximal complex distribution of the tangent bundle.
The obvious necessary condition is the integrability of the CR structure.
Although the topic is extensively studied for decades, it still contains open problems.
The discussion depends both on dimension of the CR manifold and its signature, cf. 
\cite{Andreotti1972,Kuranishi1982,Akahori1987,Webster1989}.

Analogously, the embeddability of a LC structure means a realization as a hypersurface in a para-complex space as follows.
For LC structures of dimension $2n+1$, the relevant para-complex space is $\V':=\R^{n+1,n+1}$ with the para-complex structure $\K'$ whose eigenspaces $\V'_+$ and $\V'_-$ are spanned by the first and the last $n+1$ vectors of the standard basis of $\R^{n+1,n+1}$, respectively.
The corresponding standard coordinates on $\V'_+$ and $\V'_-$ are denoted by $(v^a)$ and $(w_a)$, respectively, where $a=1,\dots,n+1$.
Let $M\subset\V'$ be a hypersurface such that the maximal $\K'$-invariant distribution in $TM$, i.e. $\D:=TM\cap\K'(TM)$, is a contact distribution.
Then the induced LC structure on $M$ is given by $E:=\D\cap\V'_+$ and $F:=\D\cap\V'_-$, which form an integrable Legendrian decomposition of $\D$. 
In particular, the leaves of $E$ and $F$ are the level sets of the coordinate functions $(w_a)$ and $(v^a)$, respectively. 

We say that a LC structure $\D=E\oplus F \subset TM$ is \emph{locally embeddable} into $\V'$ if, for every $x\in M$, there exists a local embedding from a neighbourhood $U$ of $x$ into $\V'$ such that $E=TU\cap\V'_+$ and $F=TU\cap\V'_-$.
Compared to the subtleties in the CR case, the characterization of local embeddability is quite simple:

\begin{prop} \label{prop-embed}
A LC structure is locally embeddable into a para-complex space if and only if it is integrable. 
\end{prop}

\begin{proof}
Locally embedded LC structures are necessarily integrable, as explained above.
For the reverse direction, let a LC structure $\D=E\oplus F$ on $M$ be integrable and let $(x^i,u,p_i)$, for $i=1,\dots,n$, be a local coordinates on $M$ such that $E$ and $F$ are described as in \eqref{eq-frame}.
This choice is adapted to the integrability of the distribution $F$.
Since $E$ is also integrable, the Frobenius theorem implies that, locally, there are functions $(z_a)$, for $a=1,\dots,n+1$, such that the differentials $(\d z_a)$ are linearly independent and the leaves of $E$ are the level sets of $(z_a)$.
Let the map from $M$ to $\V'=\R^{n+1,n+1}$ be given by 
\begin{align}
v^i=x^i, \quad
v^{n+1}=u, \quad
w_a=z_a,
\label{eq-embed}
\end{align}
where $i=1,\dots,n$ and $a=1,\dots,n+1$.
This map has rank $2n+1$, therefore it is an immersion. 
Its local image is a hypersurface in $\V'$, denoted as $M'$.
Accordingly, let us denote the image of $\D=E\oplus F$ as $\D'=E'\oplus F'$.
The leaves of $E'$ and $F'$ are the level sets of $(w_a)$ and $(v^a)$, respectively, thus we have $E'=TM'\cap\V'_+$ and $F'=TM'\cap\V'_-$. 
It remains to show that $\D'$ is the maximal $\K'$-invariant distribution in $TM'$.
Since $\D'=E'\oplus F'$, it is a $\K'$-invariant distribution.
Since $\D'\subset TM'$ is a contact distribution, there is no bigger $\K'$-invariant distribution.
\end{proof}

The argument in the proof is implicit by nature.
In terms of Section \ref{sec-LC-general}, an explicit realization leads to solving the system of 1st-order PDEs
\begin{align}
\parder{z}{x^i} +p_i\parder{z}{u} +f_{ij}\parder{z}{p_j} =0 ,
\label{eq-emb-PDE}
\end{align}
for $i=1,\dots,n$.
As usual, potential realizations are by no means unique.

Specifically, let us consider all $f_{ij}$ vanishing, i.e. the corresponding LC structure being locally flat.
A generating set of solutions to \eqref{eq-emb-PDE} can be given by $z_i=p_i$ and $z_{n+1}=u-x^jp_j$.
The image of the corresponding embedding \eqref{eq-embed} is then a hyperquadric in $\V'=\R^{n+1,n+1}$ given by the equation $v^jw_j-v^{n+1}+w_{n+1}=0$.
This is a possible affine realization of the model LC hyperquadric mentioned in Remark \ref{rmk:model 2.5}.

\section{Fefferman construction on LC structures}
\label{sec:Fefferman construction}

In this section, we describe the LC analogue of the classical Fefferman construction in terms of the associated Cartan geometries, cf. \cite[Section 4.5]{Cap2009}.
The primary benefit of the Cartan geometric perspective is that the construction provides a systematic framework in which we can extend results on the model to the general curved setting.
Our main contribution concerns the explicit description of a representative metric from the Fefferman conformal class associated to an integrable LC structure, see Theorem \ref{prop-Feff-integrable}.
This is obtained by suitably calibrating an adapted Cartan gauge.
Comparisons with alternative approaches and known results in the lowest dimensional case are in Remark \ref{rem-feff}.
An analogous procedure, in the case where a LC structure is induced by projective structure, is treated separately in Section \ref{LC-proj-feff}.

\subsection{Model situation} \label{sec:Feff model}
Here we adopt the notation and observations from Section \ref{sec-LC-model}.
They, in particular, yield the interpretation of the model LC manifold $M$ as the para-complex projectivization of the null-cone $\N$ in the ambient para-Hermitian space $\V=\R^{n+2,n+2}$.
In fact, this restricts to the open subset $\mc N_0=\mc N\setminus(\V_+\cup \V_-)$, since $\V_{\pm}$ does not contain any para-complex lines.
The projection map $\N_0\to M$ factors through the real projectivization of $\N_0$, which we denote by $\wt{M}$.
To be more specific, any $v\in\V$ is uniquely written as $v=v_++v_-$, where $v_\pm\in\V_\pm$.
If $v\in\N_0$, then the two projections $\N_0\to\wt{M}\to M$ compose as
\begin{align}
v_++v_- \mapsto \<v_++v_-\> \mapsto \<v_+,v_-\> ,
\label{eq-v+v-}
\end{align}
where the latter map is denoted by $\pi$.
The intermediate manifold $\wt{M}$ is just (an open subset of) the standard M\"{o}bius space, the model conformal quadric.
The conformal structure is induced by the ambient data where we forget the para-complex structure on $\V$. 
This is the model Fefferman space for LC structures.

The ambient para-complex structure $\K$ is a skew endomorphism of $\V$.
Thus, it is an element of $\wt\g=\mf{so}(n+2,n+2)$ and its exponential $\exp(\K) \in \wt G=SO(n+2,n+2)$ acts non-trivially on $\wt M$.
The typical fiber of the projection $\pi:\wt{M}\to M$ is then identified with the group $\R\times\Z_2$, whose two continuous components are parametrized as 
\begin{align}
s \mapsto \exp(s\K) , \quad
s \mapsto \K\exp(s\K) ,
\label{eq-K-expK}
\end{align}
where $s\in\R$.
For the setting as in \eqref{eq-v+v-}, the assignment $v_+=\frac1{\sqrt2}(1+k)$ and $v_-=\frac1{\sqrt2}(1-k)$ provides an identification of the fiber with two continuous branches of para-complex numbers of the norm $\pm 1$, i.e. numbers of the form 
\begin{align}
e^{k s}:=\cosh s+k\sinh s , \quad 
k e^{k s}=\sinh s+k\cosh s ,
\label{eq-versor}
\end{align}
where $s\in\R$ and $k$ is the para-complex unit ($k^2=1$).
Notably, the common asymptotes of these two curves correspond to the directions of $\V_{\pm}$. 
This way we introduce a natural (but not canonical) fiber coordinate on $\wt M$.
For later purposes, we need to make the following additional observations.

Both the model LC quadric $M$ and its Fefferman space $\wt{M}$ are  homogeneous spaces of the principal group $G=SL(n+2,\R)$.
The group $G$ acts transitively on both para-complex and real null-lines in $\N_0$. 
Hence, $M\cong G/P$ and $\wt{M}\cong G/Q$, where $P\subset G$ and $Q\subset P$ stabilize the para-complex null-line $\<e_0,e^{n+1}\>=O$ (as above) and the real null-line $\<e_0+e^{n+1}\>=:\wt{O}$, respectively.
The whole M\"{o}bius space, i.e. the real projectivization of the whole null-cone $\N$, is the homogeneous space $\wt{G}/\wt{P}$, where $\wt{G}=SO(n+2,n+2)$ and $\wt{P}\subset\wt{G}$ is the (parabolic) subgroup stabilizing $\wt{O}$.
The subgroup $Q\subset G$ is the preimage of $\wt{P}\subset\wt{G}$ under the embedding $G\subset\wt{G}$ described in \eqref{eq-eta}.
From this description and the form of the generator of $\wt O$, it follows that the previous pair of nested subgroups of $G$ has the following schematic realization
\begin{align}
Q = \pmat{a&*&*\\ 0&B&* \\ 0&0&a^{-1}} \subset \pmat{c&*&*\\ 0&D&* \\ 0&0&e} = P ,
\label{eq-Q-P}
\end{align}
where we again refer to the block decomposition as in \eqref{eq-LC-g}, in particular, $a,c,e\in\R\setminus\{0\}$, $B,D\in GL(n,\R)$ and $\det B=ce\det D=1$.
The subgroup $Q$ is normal in $P$ and so $P/Q$, the typical fiber of the projection $\pi:G/Q\to G/P$, is a group which is isomorphic to $\R\times\Z_2$, as above.
The relation to the description in \eqref{eq-K-expK}, or \eqref{eq-versor}, is given by 
\begin{align}
s \mapsto \exp
\begin{pmatrix}
s & 0 &0 \\
0 & -\frac{2 s}{n}\mbb{I}_n & 0 \\
0 & 0 &  s
\end{pmatrix}, \quad
s \mapsto 
\begin{pmatrix}
 1 & 0 & 0\\
 0 & \mbb{I}_{n-1,1} & 0\\
 0 & 0 & -1
\end{pmatrix}
 \exp
\begin{pmatrix}
s & 0 &0 \\
0 & -\frac{2 s}{n}\mbb{I}_{n} & 0 \\
0 & 0 & s
\end{pmatrix} ,
\label{eq-P-Q}
\end{align}
where $s\in\R$, $\mbb{I}_n$ is the identity matrix of rank $n$, and $\mbb{I}_{n-1,1}$ is the diagonal matrix of signature $(n-1,1)$. 
Note that the union of the two curves in \eqref{eq-P-Q} gives a preferred (but not canonical) global section of $P\rightarrow P/Q$.
In this way, the fibers of the projection $\pi$ are seen as the orbits of these elements on $\wt M\cong G/Q$.

Denoting by $\wt\G$ the preimage of $G/Q\subset \wt G/\wt P$ under the principal $\wt P$-projection $\wt G\to\wt G/\wt P$, the actual understanding of the model Fefferman construction can be sketched as follows:
\begin{align}
\begin{split}
\xymatrix@R=.7\baselineskip{
& \wt\G \ar[dd] \\
G \ar[dd] \ar[dr] \ \ar@{^{(}->}[ur] & \\
& \wt M\cong G/Q  \ar[dl]^\pi \\
M\cong G/P & \\
}
\end{split}
\label{dia-feff-extension}
\end{align}

\subsection{General scheme} \label{LC-feff-rough}
The Fefferman construction of a LC structure modifies the previous homogeneous picture \eqref{dia-feff-extension} to a general LC Cartan geometry of type $(G,P)$.
The general approach we follow is developed in \cite{Cap2006} and \cite{Cap2009}, details related to our present situation are adapted from \cite{Hammerl2017}.

Let $(\G\to M,\om)$ be a Cartan geometry of type $(G,P)$ with an underlying LC structure on $M$. 
Let $Q\subset P$ be as above. 
The \emph{Fefferman space} of $M$ is the corresponding quotient space $\wt{M}:=\G/Q$.
The Cartan connection $\om$ on $\G$ defines the Cartan geometry of type $(G,Q)$ over $\wt{M}$. 
The embedding $G\subset\wt{G}$ produces an equivariant extension of the previous Cartan geometry to a Cartan geometry $(\wt\G\to\wt M,\wt\om)$ of type $(\wt G,\wt P)$ as follows:
denoting the embedding by $\eta$ as in \eqref{eq-eta}, the restriction $\eta|_Q:Q\to\wt{P}$ is behind the extension of the principal bundle $\wt\G:=\G\times_Q\wt{P}$, and its derivative $\eta':\g\to\wt{\g}$ extends the values of the Cartan connection $\wt\om:=\eta'\o\om$.
The underlying structure of the latter Cartan geometry is the induced conformal structure on the Fefferman space.
Thus, the Fefferman construction with the groups specified above produces a conformal Cartan geometry from a given LC Cartan geometry.

As an associated bundle, the Fefferman space can be described as 
\begin{align}
\wt M \cong \G\times_P P/Q .
\label{eq-feff-PQ}
\end{align}
Its tangent bundle can be seen in two ways
\begin{align}
T\wt{M} \cong \G\times_Q \g/\q \cong \wt\G\times_{\wt{P}} \wt\g/\wt\p ,
\label{eq-feff-TM}
\end{align}
cf. \eqref{eq-TM}.
The conformal structure on $\wt M$ corresponds to the standard inner product on  $\wt\g/\wt\p$,
i.e. the unique $\wt{P}$-invariant inner product, up to a non-zero multiple.
This corresponds to an inner product on  $\g/\q$ that is $Q$-invariant up to a non-zero multiple.
With these preparations, we can interpret the Fefferman space and the induced conformal structure in terms of the underlying LC data.

\begin{prop} \label{prop-feff-rough}
Let $(\G\to M,\om)$ be a parabolic geometry with the underlying LC structure $\D=E\oplus F\subset TM$, and let $(\wt\G\to\wt{M},\wt\om)$ be the conformal parabolic geometry obtained by the Fefferman construction.
Then the Fefferman space $\wt{M}$ is identified as
\begin{align}
\wt M \cong \mc{P}\big( (\wedge^n E \oplus\wedge^n F) \setminus (\wedge^n E\cup\wedge^n F) \big) ,
\label{eq-feff-space}
\end{align}
where $\mc{P}$ denotes the (real) projectivization and $n=\op{rank} E=\op{rank} F$.
The conformal structure on $\wt{M}$ is represented by the metric corresponding to the quadratic form on $\g/\q$ given by 
\begin{align}
\bmat{z&*&* \\ X&*&* \\w&Y^t&z} \mapsto Y^tX+2zw .
\label{eq-feff-form}
\end{align}
\end{prop}

\begin{proof}
With the conventions beneath the formula \eqref{eq-LC-g}, the Legendrian subbundles $E,F\subset\D$ correspond to $P$-modules $\g_{-1}^E, \g_{-1}^F\subset\g_{-1}$.
The induced $P$-action on respective top exterior powers, hence on $\mc{P}(\wedge^n\g_{-1}^E \oplus\wedge^n\g_{-1}^F)$ is as follows.
For any $\eta\in\wedge^n\g_{-1}^E$ and $\ph\in\wedge^n\g_{-1}^F$, the element of $P$ taken from \eqref{eq-Q-P} acts as
\begin{align}
\<\eta + \ph\> \mapsto 
\<(c^{-n}\det D) \eta + (e^n\det D^{-1}) \ph\> =
\<\eta + (\det D^{-n-2})\ph\> ,
\label{eq-P-act}
\end{align}
where we have used $ce\det D=1$.
Thus, in the case $n$ is odd, $P$ acts transitively on 
\begin{align}
\label{eqn-feff-Lie}
\mc{P}\left((\wedge^n \mathfrak{g}_{-1}^E \oplus\wedge^n \mathfrak{g}_{-1}^F) \setminus (\wedge^n \mathfrak{g}_{-1}^E \cup\wedge^n \mathfrak{g}_{-1}^F)\right) 
\end{align}
and the stabilizer of any element in \eqref{eqn-feff-Lie} is given by $\det D=1$.
This condition determines exactly the subgroup $Q\subset P$, cf. \eqref{eq-Q-P}.
Hence the set in \eqref{eqn-feff-Lie} is the homogeneous space $P/Q$.
This, together with \eqref{eq-feff-PQ}, yields \eqref{eq-feff-space}.

In the case $n$ is even, the action above is not transitive but it can be extended as follows.
Let us consider the transformation of \eqref{eqn-feff-Lie} given by 
\begin{align}
\<\eta+\ph\> \mapsto \<\eta-\ph\> .
\label{eq-Z2-act}
\end{align}
This is a map of order 2 that commutes with the action of $P$. 
Let us consider the group  $P':=P\times\Z_2$, whose action on \eqref{eqn-feff-Lie} is given by the component actions \eqref{eq-P-act} and \eqref{eq-Z2-act}, respectively. 
(Here we consider $\Z_2$ as the multiplicative group $\{1,-1\}$ and \eqref{eq-Z2-act} describing the action of the element $-1$.)
This action is transitive and the stabilizer of any element is $Q':=Q_{\pm}\times\{1\}\subset P'$, where $Q_{\pm}\subset P$ is given by $\det D=\pm 1$.
Hence the set in \eqref{eqn-feff-Lie} is the homogeneous space $P'/Q'$ which, however, can be identified with $P/Q$ as follows.
Consider the map $P\to P'$ given by $p\mapsto (p,s(p))$,
where $s(p)$ denotes the sign of the determinant of $D$, the middle block of $p\in P$ as above.
This map is a group homomorphism  which induces an isomorphism of the factors $P/Q\cong P'/Q'$.

Altogether, the interpretation \eqref{eq-feff-space} holds independently of the parity of $n$.

For the second part of the proposition, one can verify that \eqref{eq-feff-form} is the only quadratic form on $\g/\q$ that is $Q$-invariant, up to a non-zero multiple.
Thus, together with \eqref{eq-feff-TM}, the claim follows.
Alternatively, one can use the $Q$-module isomorphism $\wt\g/\wt\p\cong\g/\q$ for rewriting the standard inner product on $\wt\q/\wt\p$.
The isomorphism is induced by the Lie algebra homomorphism $\g\to\wt\g$ corresponding to \eqref{eq-eta}.
Regarding the details, one needs explicit matrix realizations of the homomorphism (which is easy) and the respective subalgebras (which is a bit cumbersome).
Details for both potential approaches can be found in \cite{Hammerl2017}.
\end{proof}

There is a distinguished generator of the vertical subbundle of the projection $\pi:\wt{M}\to M$, namely, 
the fundamental vector field of the right action of the 1-parameter subgroup of $P/Q\cong\R\times\Z_2$, given in the left part of display \eqref{eq-P-Q}.
It can also be described as the projection of the vector field on  the total space of the bundle $p:\wt\G\to\wt M$, which corresponds to the para-complex structure $\K$ on $\V$ as follows.
Since $\K$ is an element of $\wt\g=\mf{so}(n+2,n+2)$, it defines a constant vector field on $\wt\G$, denoted by $\wt\om^{-1}(\K)$. 
Restricting to the image of the canonical inclusion $\G\subset\wt\G$, the vector field is $Q$-invariant, hence the projection 
\begin{align}
\k:=Tp(\wt\om^{-1}(\K)\vert_{\mathcal{G}}) 
\label{eq-K}
\end{align}
is well defined. 
This is a nowhere vanishing null-vector field on $\wt{M}$ whose orthogonal complement $\k^\perp\subset T\wt{M}$ projects to the contact distribution $\D\subset TM$.
In fact, there are finer data on $\wt{M}$ which allow to characterize the conformal structure obtained via the Fefferman-type construction, see \cite[Section 3.6]{Hammerl2017}.
In particular, there are null distributions $\wt{E},\wt{F}\subset T\wt{M}$ such that $\wt{E}+\wt{F}=\k^\perp$, $\wt{E}\cap\wt{F}=\<\k\>$ and which project to the Legendrian decomposition $E\oplus F=\D$.

The key question for constructions of the current type is whether (or in which cases) they preserve the normality condition.
Analogously to the original CR situation, 
this happens if and only if the initial Cartan geometry is torsion-free.
For our purposes, we also emphasize that, in the integrable case, the distinguished vector field $\k$ is a conformal Killing field.
Both these facts are first stated in \cite{Cap2008a} in the context of CR geometry and later adapted to the LC setting in \cite{Hammerl2017}.
The normality issues are typically subtle, while the fact that $\k$ is an infinitesimal conformal symmetry follows easily from the construction.

\begin{prop}[{\cite[Theorems 2.5 and 3.1]{Cap2008a}, \cite[Section 3.6]{Hammerl2017}}] \label{prop-feff-normal}
Let $(\G\to M,\om)$ be the normal parabolic geometry with the underlying LC structure and let $(\wt\G\to\wt{M},\wt\om)$ be the conformal parabolic geometry obtained by the Fefferman construction.
Then $\wt\om$ is normal if and only if $\om$ is torsion-free, i.e. the LC structure is integrable.
In such case, the vector field \eqref{eq-K} is a nowhere vanishing conformal Killing field. 
\end{prop}

Since $\k$ is nowhere vanishing, there exists a metric in the conformal class for which $\k$ is a true Killing field.
This fact is made explicit in Section \ref{LC-feff-explicit} and then applied in Section \ref{sec-Kropina}.

\subsection{Explicit Fefferman metric for integrable LC structures} \label{LC-feff-explicit}
Here we merge the previous rough observations with a finer analysis of the normal Cartan connection associated to an integrable LC structure.
This yields an explicit local description of a representative metric from the Fefferman conformal class.

Throughout this section, let $U\subset M$ be a sufficiently small open subset.
With the notation as above, for  a section $\phi:U\to\G$ of the Cartan bundle projection, we denote the components of the corresponding gauge $\phi^*\om:TU\to\g$ by
\begin{align}
\begin{pmatrix}
\om^0_0 & \om^0_j & \om^0_{n+1} \\
\om^i_0 & \om^i_j & \om^i_{n+1} \\
\om^{n+1}_0 & \om^{n+1}_j & \om^{n+1}_{n+1}
\end{pmatrix} ,
\label{eq-gauge}
\end{align}
where we adopt the conventions from \eqref{eq-LC-g} again.
Note that $\sum\limits_{r=0}^{n+1} \om^r_r =0$.
Without any loss of generality, one can always choose a local section $\phi$ so that the $\g_-$-part of \eqref{eq-gauge} is formed by an adapted coframe \eqref{eq-coframe}, 
\begin{align}
\om^{n+1}_0=\si,\quad \om^i_0=\th^i,\quad \om^{n+1}_j=\pi_j , 
\label{eq-gauge0}
\end{align}
cf. \cite[Lemma 2.8]{Doubrov2020}.
The remaining components of \eqref{eq-gauge} are, in principle, deducible from the normality condition and expressible in terms of the defining functions  $f_{ij}$ and their derivatives.
Passing from $M$ to $\wt{M}$, the Fefferman metric is schematically given by \eqref{eq-feff-form}.
This shows that only part of the $\g_0$-block of the Cartan gauge is needed.

The previous setting refers to some coordinates $(x^i,u,p_i)$ on $U$.
An additional fiber coordinate on the Fefferman space $\wt{M}$ is introduced as follows.
Composing the gauge section $\phi:U\to\G$ with the bundle projection $\G\to\wt{M}$, we get a local section of the Fefferman fibration $\wt{M}\to M$. 
This is a preferred (but not canonical) section that is fully determined by the previous choices. 
Any other element of $\wt{M}$ over $U$ can be related to the image of this section by the action of the group $P/Q\cong\R\times\Z_2$, which is described in \eqref{eq-P-Q}.
This induces the coordinate $s\in\R$ on (each of the two continuous components of) the fiber. 
With this setup, we derive the local formula for a representative metric of the Fefferman conformal class: 

\begin{theorem} \label{prop-Feff-integrable}
Let $\D=E\oplus F$ be an integrable LC structure on a manifold $M$ of dimension $2n+1$ and let 
\begin{align*}
\si=\d u -p_i \d x^i, \quad
\th^i=\d x^i, \quad
\pi_i=\d p_i -f_{ij}\d x^j
\end{align*}
be an adapted coframe, in a local coordinate $(x^i,u,p_i)$, such that $E=\ker\<\si,\pi_i\>$ and $F=\ker\<\si,\th^i\>$.
Let $[\wt{g}]$ be the induced conformal structure on the Fefferman space $\wt{M}$ and let $s$ be the induced fiber coordinate of the projection $\wt{M}\to M$.
Then a representative metric from the conformal class has the form
\begin{align}
\wt{g} = \theta^i\odot \pi_i + \sigma\odot \varpi ,
\label{eq-Feff-integrable}
\end{align}
where
\begin{align}
\varpi = \frac1{n+2} \sum_{i,j=1}^{n} \left( - \frac1{n+1}\frac{\del^2 f_{ij}}{\del p_i\del p_j}\,\si - 2\frac{\del f_{ij}}{\del p_i}\,\th^j  \right)+ 2\d s . 
\label{eq-varpi}
\end{align}
\end{theorem}

\begin{proof}
Let $\psi=q\o\phi:U\to\wt{M}$ be the preferred section of the Fefferman projection, where $q:\G\to\wt{M}$.
Along the image of $\psi$, the form of the metric in \eqref{eq-Feff-integrable} follows from \eqref{eq-feff-form} and the indicated substitutions, where
\begin{align*}
\varpi=\om^0_0+\om^{n+1}_{n+1}=-\sum\limits_{i=1}^{n} \om^i_i
\end{align*}
is the only term to be specified.
For integrable LC structures, it follows that the curvature $\Om$ of the normal Cartan connection takes values in the $P$-submodule $\g_0^{ss}\oplus\g_1\oplus\g_2\subset\g$, where $\g_0^{ss}$ is the semisimple part of $\g_0$, see \cite[Section 3.8]{vcap2009}.
Denoting the components of $\phi^*\Om$ as in \eqref{eq-gauge}, we have 
\begin{align}
\Om^r_s=\d\om^r_s+\om^r_t\wedge\om^t_s , 
\label{eq-K-gauge}
\end{align}
where $0\le r,s,t\le n+1$.
The previous restriction on the value of $\Omega$ implies
\begin{align}
\Om^{n+1}_0=\Om^i_0=\Om^{n+1}_j = \Om^0_0=\Om^{n+1}_{n+1}=\sum\limits_{i=1}^n \Om^i_i=0 ,
\label{eq-KKK}
\end{align}
for $1\le i,j\le n$.
By expanding and inspecting the conditions \eqref{eq-KKK} using \eqref{eq-K-gauge}, one gradually gets specifications on the components $\om^r_s$.
Typically, this is a tedious process. 
Since we are only concerned with (some of) the diagonal terms, we end up relatively quickly with the formula
\begin{align*}
\varpi = \frac1{n+2} \sum_{i,j=1}^{n} \left( - \frac1{n+1}\frac{\del^2 f_{ij}}{\del p_i\del p_j}\,\si - 2\frac{\del f_{ij}}{\del p_i}\,\th^j \right) . 
\end{align*}
The same outcome can also be deduced from \cite[Section 2.5]{Doubrov2020}, where similar reasoning (with different aims) is presented in detail.
See, in particular, equations (2.6) and (2.12) in that reference.

To obtain the general formula of $\wt{g}$ outside the image of $\psi$, we examine how the gauge components change along the fibration $\wt M\to M$.
This is controlled by the action of elements from \eqref{eq-P-Q}, whose substitution into the transformation formula \eqref{eq-gauge-change} gives, after a simple calculation, the expressions \eqref{eq-Feff-integrable} and \eqref{eq-varpi}.
\end{proof}

\begin{remarks} \label{rem-feff}
(i)
The conformal invariance of \eqref{eq-Feff-integrable} follows by the very construction.
It can also be checked directly by passing to another adapted coframe, expanding the key conditions \eqref{eq-KKK} according to \eqref{eq-K-gauge} and repeating the described procedure.
With the choice as in \eqref{eq-coframe-change}, the rescaled metric changes by the factor $e^{2f}$.

(ii)
The conformal Killing field $\k$, introduced abstractly in \eqref{eq-K}, is described in current local coordinates as 
$\k = \frac{\del}{\del s}$.
It is easy to see that this is a true Killing field of the representative metric \eqref{eq-Feff-integrable}.

(iii)
The formula \eqref{eq-Feff-integrable} involves only the $\g_-$-part and the $\g_0$-part of the normal Cartan gauge determined by a choice of adapted coframe. 
It can therefore be interpreted in terms of exact Weyl connections, the distinguished affine connections preserving the LC structure and the contact form.
To compare with the classical Fefferman construction for CR structures, we refer primarily to \cite{Lee1986} where the formulas are derived via another distinguished class of connections, namely, the Webster connections.
Both the Weyl and the Webster connections are defined for any parabolic contact structure and they can be related in a rather convenient way, see \cite[Sections 5.3.12--14]{Cap2009}.
This provides a hint to other potential reformulations of \eqref{eq-Feff-integrable}. 
Still, we regard our present approach to be the most straightforward.

(iv)
LC structures of dimension 3, i.e. for $n=1$, are automatically integrable and their normal Cartan  curvature is much easier to analyze than for $n>1$.
However, the expression \eqref{eq-Feff-integrable} is unchanged, only simplified and it can already be found in \cite{Nurowski2003} and \cite{Bor2022}.
More details on this special case are in Section \ref{sec-dim3}.
\end{remarks}

\section{Canonical curves} \label{sec:curves}
The aims of this section are as follows. 
Firstly, we give an elementary interpretation of chains and null-chains in the homogeneous model based on Section \ref{sec-LC-model}.
Referring to this description and the notion of development of curves, we then specify the correspondence between (null-)chains on a LC manifold and null-geodesics of the conformal Cartan bundle obtained from the Fefferman construction.
In the case the LC structure is integrable, the latter curves are just the null-geodesics of the Fefferman metric.
This, in particular, leads to the description of chains as the extremal curves of the Kropina functional in Theorem \ref{prop-Kropina}.
Before we go into the details, we first recall some general theory.

\subsection{Preliminaries} \label{curves-prelim}
There are several ways to define, and hence study, canonical curves in particular geometries.
In the framework of Cartan geometries, general definitions reflect the model situation.
In the homogeneous model, $G\to G/P$, the canonical curves are the projections of shifted 1-parameter subgroups of the principal group $G$.
I.e., they are the curves parametrized as $t\mapsto g\exp(tX)/P$, where $g\in G$, $X\in\g$ and $t\in\R$.
According to the type of the generator $X\in\g$, various types of canonical curves can be categorized.
For parabolic Cartan geometries, we restrict to the negative part of the grading $\eqref{eqn:parabolic grading alg}$ of $\g$.
The types of curves are then distinguished by subsets $S\subseteq\g_-$ which
(in order for the concept to have an invariant geometric meaning)
are assumed to be $G_0$-invariant, where $G_0\subset P$ is the Lie group with the Lie algebra $\g_0\subset\p$.

A generalization of the notion of canonical curves to general Cartan geometries $(\G\to M,\om)$ of type $(G,P)$ can be done as follows.
Let us consider the associated bundle $\mc{S}:=\G\times_P G/P$ over $M$ with its canonical section 
corresponding to the origin in $G/P$.
Along this section, the vertical subbundle is naturally identified with $\G\times_P \g/\p$, i.e. the tangent bundle $TM$.
This makes precise the intuitive idea of ``osculating'' $M$ at each point by the model space $G/P$, the idea that may be spotted already in Cartan's work.
The Cartan connection on $\G$ induces (after a $G$-extension) a connection on $\mc{S}$, which allows one to define the \textit{development of curves} on $M$ into the homogeneous space $G/P$.
More specifically, for a curve $c:I\to M$ and its canonical lift $\hat c:I\to\mc S$,
one may parallel transport the points $\hat c(t)$, for varying $t\in I$, into a chosen fiber over $x=c(t_0)$.
Locally, this draws a curve in $\mc S_x\cong G/P$ passing through the origin. 
For fixed choices, this construction provides a bijection between germs of curves through $x$ in $M$ and germs of curves through origin in $G/P$ that preserves the order of contact.
This is an important tool for  studying curves on $M$ via their counterparts in $G/P$.

\begin{definition} \label{def-curves}
Let $(\G\to M,\om)$ be a parabolic geometry of type $(G,P)$ and $S\subseteq\g_{-}$ be a $G_0$-invariant subset. 
A curve in $M$ is called a \emph{canonical curve of type $S$} if it admits a development into the curve in $G/P$ of the form 
$$t\mapsto\exp(tX)/P,$$ 
where $X\in S$ and $t\in I$.
Canonical curves (of certain type) of a parabolic geometric structure are the canonical curves of the corresponding regular and normal parabolic geometry.
\end{definition}

Equivalently, canonical curves are projections of flow lines of constant vector fields from the principal Cartan bundle.
More precisely, a curve $c:I\to M$ is a canonical curve of type $S$ if and only if it admits a lift $\bar{c}:I\to\G$ such that the value of $\om\left(\frac{d}{dt} \bar{c}(t)\right)$ is constant, for all $t\in I$, and belongs to $S$.

By definition, canonical curves are endowed with families of admissible reparametrizations, which  turn out to be either projective or affine.
Usually, we consider the curves as unparametrized ones.
Details on the the general theory of canonical curves, as well as particular applications, can be found in \cite{Cap2004} or \cite[Section 5.3]{Cap2009}.

\subsection{Model interpretations} \label{curves-model}

Here we describe chains and null-chains in the LC homogeneous model, referring to the observations from Section \ref{sec-LC-model}.
For each type of canonical curves, tangent vectors in any point are of the corresponding type, respectively.
They are classified in Lemma \ref{lem:TOM}, which implicitly enters the discussion that follows.
Results of this section are used later in Section \ref{curves-correspondence}, but they are also of interest on their own.
Primarily, we give an interpretation in terms of the model ``para-hyperquadric'', as commented in Remark \ref{rmk:model 2.5}, and its intersections with appropriate subspaces.
This is the closest analogue of well-known descriptions of canonical curves in model CR and conformal geometry, see e.g. \cite{Eastwood2020} for a survey and references.
We also offer an underlying projective interpretation.
This extends notably the only attempt in this direction we are aware of, \cite[Proposition 4.1]{Bor2022}, where a description of model chains in the 3-dimensional case is given.

By definition, \textit{chains} of a LC parabolic geometry are the canonical curves of type $S=\g_{-2}$, cf. the grading description in \eqref{eq-LC-g}.
As unparametrized curves, chains are uniquely given by a tangent direction in one point, provided it is transverse to the contact distribution.
Chains carries a natural projective family of parametrizations.

\begin{prop} \label{prop:model-chain}
Let $M$ be the homogeneous model of LC structure of dimension $2n+1$.
\begin{enumerate}[(a)]
\item \label{chain-a}
Interpreting $M$ as the para-complex projectivization of the null-cone $\N\subset\R^{n+2,n+2}$, 
take a para-complex plane $U\subset\R^{n+2,n+2}$ that is non-degenerate.
Then the para-complex projectivization of $\N\cap U$ is a chain in $M$ and all chains are of this form.
\item \label{chain-b}
Interpreting $M$ as the set of incident pairs of points and hyperplanes in real projective space $\R\P^{n+1}$,
take a disjoint pair of line $\ell$ and subspace $L$ of codimension 2 in $\R\P^{n+1}$.
Then the set of all incident pairs $(p,H)\in M$ such that $p\in\ell$ and $H\supset L$ is a chain in $M$ and all chains are of this form.
\end{enumerate}
\end{prop}

\begin{proof}
Dealing with the homogeneous model $M=G/P$, it suffices to analyze the situation in a particular point and a particular direction transverse to the contact distribution.
Let $A=\spmat{0&0&0\\0&0&0\\1&0&0}\in\g_{-2}$ and $c(t)=\exp(tA) / P$ be the chain passing through the origin in $G/P$ in the direction of $A/\p\in\g/\p$, i.e.
\begin{align*}
c(t)=\pmat{1&0&0\\0&\mbb{I}&0\\t&0&1} \big/ P ,
\end{align*}
where $\mbb{I}$ is the identity matrix of rank $n$.

Under the interpretation \eqref{chain-a}, the origin and the tangent vector correspond to $O=\<e_0,e^{n+1}\>$ and  $A = e^0\otimes (e_{n+1}/O) - e_{n+1}\otimes (e^0/O)\in O^*\otimes\V/O$, respectively, cf. the identification \eqref{eq-TOM}.
The chain itself corresponds to 
\begin{align}
c(t)=\<e_0+t e_{n+1},-t e^0+e^{n+1}\> .
\label{eq-chain-a}
\end{align}
The chain \eqref{eq-chain-a} consists of para-complex null-lines contained in the subspace 
\begin{align}
U = O+\op{im}(A) = \<e_0,e_{n+1},e^0,e^{n+1}\> ,
\label{eq-2-2}
\end{align}
which is evidently a non-degenerate para-complex plane.
Conversely, there are two 1-parameter families of real null-planes in $U$ (known as $\al$- and $\be$-planes, respectively).
One of them is the chain \eqref{eq-chain-a}, the other can be parametrized as
$\<e_0+s e^{n+1},-s e^0+e_{n+1}\>$.
However, members of the latter family are not para-complex lines.
Hence the intersection $\N\cap U$ is just the chain \eqref{eq-chain-a}.

Under the interpretation \eqref{chain-b}, the chain \eqref{eq-chain-a} corresponds to the curve of incident pairs
\begin{align}
p(t)=\<e_0+t e_{n+1}\>, \quad H(t)=\<\ker (-t e^0+e^{n+1})\>
\label{eq-chain-b}
\end{align}
of points and hyperplanes in $\R\P^{n+1}$.
The points $p(t)$ form a line $\ell$ and the hyperplanes $H(t)$ have a common intersection $L$, namely, the projectivization of the subspace $\<e_1,\dots,e_n\>\in\R^{n+2}$.
Clearly, $\ell$ and $L$ are disjoint and the dimension of $L$ is $n-1$.
The converse statement follows from the previous interpretation.
Directly, given $\ell$ and $L$, any $p\in\ell$ determines an incident hyperplane $H=p+L$.
Such family allows a parametrization as in \eqref{eq-chain-b}, which is a chain.
\end{proof}

Any chain in $M$ can always be parametrized analogously to \eqref{eq-chain-a}. 
In particular, both components of the intersections with $\pm1$-eigenspaces are (standardly parametrized) affine lines.
This pair of lines, however, is not arbitrary. 
This detail is to be compared with an upcoming discussion for null-chains.

By definition, \textit{null-chains} of a LC parabolic geometry are the canonical curves of type $S\subset\g_{-1}$, where $S$ is $\g^E_{-1}$, $\g^F_{-1}$ or the set of generic null-vectors in $\g_{-1}$, i.e. those null-vectors that are not contained in $\g^E_{-1}\cup\g^F_{-1}$. 
Accordingly, we distinguish null-chains of type $E$, $F$ and the generic ones.
As a parametrized curve, any null-chain is uniquely given by its 2-jet and carries a natural projective family of parametrizations.
As unparametrized curves, null-chains of type $E$ or $F$ are uniquely given by a tangent direction in one point, whereas there is a 1-parameter family of generic null-chains obeying an analogous initial condition, cf. \cite[Proposition 5.3.8]{Cap2009}.
We primarily focus on generic null-chains and treat the other types as degenerate cases:

\begin{prop} \label{prop:model-null-chain}
Let $M$ be the homogeneous model of LC structure in dimension $2n+1$.
\begin{enumerate}[(a)]
\item \label{nchain-a}
Interpreting $M$ as the para-complex projectivization of the null-cone $\N\subset\R^{n+2,n+2}$, 
take a null para-complex plane $U\subset\N$ and its intersections $U_\pm:=U\cap\V_\pm$ with the eigenspaces of the para-complex structure on $\R^{n+2,n+2}$.
Then any pair of real affine lines in $U_+$ and $U_-$ determines a generic null-chain in $M$ and all generic null-chains are of this form.
Null-chains of type $E$, respectively $F$, are degenerate cases of generic ones, where 
$\dim_\R U=3$ and $\dim_\R U_+=1$, respectively $\dim_\R U_-=1$.
\item \label{nchain-b}
Interpreting $M$ as the set of incident pairs of points and hyperplanes in real projective space $\R\P^{n+1}$,
take a line $\ell$ incident with a subspace $L$ of codimension 2 in $\R\P^{n+1}$.
Then a set of incident pairs $(p,H)\in M$ such that $p\in\ell$ and $H\supset L$ is a generic null-chain in $M$ and all generic null-chains are of this form.
Null-chains of type $E$, respectively $F$, are degenerate cases of generic ones, where all hyperplanes $H\supset L$, respectively all points $p\in\ell$, coincide.
\end{enumerate}
\end{prop}
\begin{proof}
The proof is analogous to the one of Proposition \ref{prop:model-chain}.
In particular, we analyze the situation in one point and one null-direction from the contact distribution.
Let $A=\spmat{0&0&0\\e_1&0&0\\0&e^n&0}\in\g_{-1}$ be a generic null-vector in $\g_{-1}$ and $c(t)=\exp(tA)/P$ be a null-chain passing through the origin in $G/P$ in the direction of $A/\p\in\g/\p$, i.e.
\begin{align*}
c(t)=\pmat{1&0&0\\te_1&\mbb{I}&0\\0&te^n&1} \big/ P.
\end{align*}

Under the interpretation \eqref{nchain-a}, the origin, the tangent vector and the null-chain itself corresponds to $O=\<e_0,e^{n+1}\>$,  $A = e^0\otimes (e_1/O) - e_{n+1}\otimes (e^n/O)\in O^*\otimes\V/O$ and 
\begin{align}
c(t)=\<e_0+t e_1,-t e^n+e^{n+1}\> ,
\label{eq-nchain-a}
\end{align}
respectively.
By the action of the subgroup $\exp\p_+ \subset P$ on the model null-chain \eqref{eq-nchain-a}, it follows that all null-chains passing through the origin with the same tangent vector as \eqref{eq-nchain-a} are given as 
\begin{align}
\hat c(t)=\< (1+at)e_0+te_1, -te^n+(1+bt)e^{n+1} \> , 
\label{eq-nnchain-a}
\end{align}
where $a,b\in\R$ are free parameters.
The null-chain \eqref{eq-nnchain-a} lies in the subspace
\begin{align}
U = O+\op{im}(A) = \<e_0,e_1,e^n,e^{n+1}\> ,
\label{eq-null}
\end{align}
which is evidently a null para-complex plane, 
and the intersections with $\pm1$-eigenspaces draw affine lines in $U_\pm=U\cap\V_\pm$.
The freedom controlled by $a,b\in\R$ (and potential projective reparametrizations) shows that any pair of affine lines in $U_\pm$ can be achieved this way.

Under the interpretation \eqref{nchain-b}, the null-chain \eqref{eq-nchain-a} corresponds to the curve of incident pairs
\begin{align}
p(t)=\<e_0+t e_{1}\>, \quad H(t)=\<\ker (-t e^n+e^{n+1})\>
\label{eq-nchain-b}
\end{align}
of points and hyperplanes in $\R\P^{n+1}$.
The points $p(t)$ form a line $\ell$ and the hyperplanes $H(t)$ have a common intersection $L$, namely, the projectivization of the subspace $\<e_0,\dots,e_{n-1}\>\in\R^{n+2}$.
Clearly, $\ell$ and $L$ are incident and $\dim L=n-1$.
Passing to a general null-chain \eqref{eq-nnchain-a} and using the previous justification, we see the argument is complete, i.e. any incidence configuration of the present type can be achieved this way.

The specification of null-chains of type $E$, respectively $F$, follows immediately by substituting globally 0 instead of $e^n$, respectively $e_1$.
\end{proof}

There are two free parameters $a,b\in\R$ in the expression \eqref{eq-nnchain-a}.
But, for $a=b$, this expression is just a (projective) reparametrization of \eqref{eq-nchain-a}.
Moreover, the pairs $(a,b)$ and $(a',b')$ determine the same curve up to an (affine) reparametrization if and only if $a'=ka$ and $b'=kb$, for some $k\in\R$.
Hence, indeed, we are in concordance with the above stated counts, i.e. there is a 1-parameter family of unparametrized generic null-chains passing through a given point in a given direction.

\smallskip

With the preceding model interpretations, we may give a simple criterion for local connectivity of points in the homogeneous model by chains and null-chains, respectively.
Let $M$ be the model LC manifold interpreted as in claim \eqref{LC-model-e} of Proposition \ref{prop-LC-model}, i.e. as the para-complex projectivization of the null-cone in $\R^{n+2,n+2}$.
From Lemma \ref{lem:TOM}, we know that, for any $L\in M$ and any transverse vector $w\in T_LM \setminus \D_L$, the corresponding subspace $\op{im}(w)+L$ of the ambient vector space $\R^{n+2,n+2}$ is a non-degenerate para-complex plane. 
From claim \eqref{chain-a} of Proposition \ref{prop:model-chain}, we know that such subspace determines a chain uniquely (as unparametrized curves), hence play a role of initial condition.
Similarly, for any  $L\in M$ and a generic null-vector from $w\in\D_L$, the corresponding subspace $\op{im}(w)+L$ is a null para-complex plane.
From claim \eqref{nchain-a} of Proposition \ref{prop:model-null-chain}, we know that such a subspace determines a family of generic null-chains.
In both cases, the curve is contained in the para-complex projectivization of the initial subspace.
Passing to the boundary condition instead, we get the following characterization:

\begin{corollary} \label{cor:chain-connect}
Let $M$ be the homogeneous model of LC structure interpreted as the para-complex projectivization of the null-cone in $\R^{n+2,n+2}$. 
Let $L_1$ and $L_2$ be two distinct para-complex null-lines in $\R^{n+2,n+2}$.
Then $L_1$ and $L_2$ can be connected by a chain, respectively a generic null-chain, if and only if the subspace $L_1+L_2\subset\R^{n+2,n+2}$ is a para-complex plane that is non-degenerate, respectively null.
\end{corollary}

In particular, given points cannot be connected by chain and null-chain simultaneously.
Generically, two points are connected by a chain, which is unique. 
Considering the connecting chain as a segment, there are actually two solutions, since chains in the homogeneous model are closed curves.
If given points are connected by a generic null-chain, then there is a 1-parameter family of such.
Taking into account also null-chains of type $E$ or $F$, one easily formulates an analogous characterization using part (a) of Lemma \ref{lem:TOM}.
As usual, this can be seen as a degenerate case of the generic null-chain situation.
This degeneracy implies that the connecting family of curves, provided it exists, collapses to a single curve.

\subsection{Canonical curves under Fefferman correspondence} \label{curves-correspondence}

Here we describe the correspondence between canonical curves under the Fefferman correspondence.
Primarily, we are interested in null-geodesics of the Fefferman space $\wt M$ which, as we show, project to chains, null-chains or points on the underlying LC manifold $M$.
By the very construction and the notion of development of curves, the full analysis can be done in the model situation, cf. Sections \ref{LC-feff-rough} and \ref{curves-prelim}.
This comprises an analysis of the correspondence between projections of 1-parameter subgroups (of particular types) of the principal groups, which translates directly to the general curved setting.
Instead of a formal description of 1-parameter subgroups in $G$ and $\wt G$ and their projections, we use our LC model interpretations from Section \ref{curves-model} and the standard conformal ones, respectively.

We utilize the setting from Section \ref{sec:Feff model}, where $\wt{M}$ and $M$ stays for the real and the para-complex projectivization, respectively, of the (open subset of the) null-cone $\mc N_0=\mc N\setminus(\V_+\cup \V_-)$.
The model projection $\pi:\wt{M}\to M$ is described as 
\begin{align}
\<v_++v_-\> \mapsto \<v_+,v_-\> ,
\label{eq-pi}
\end{align}
where $v=v_++v_-$ is the unique decomposition corresponding to $\V=\V_+\oplus\V_-$ and $v\in\mc N_0$, i.e. $v_\pm\ne 0$ and $v_+\.v_-=0$.
Note that this map is just $\<v\>\mapsto\<v,\K(v)\>$, where $\K$ is the para-complex structure on $\V$.
I.e., for any real null-line not contained in $\V_\pm$, the map $\pi$ returns its para-complex hull.
More generally, for a vector subspace $W\subset\V$, its para-complex hull is denoted by $\K(W)$.

The tangent map of the projection $\pi$ is described as follows.
On the one hand, for any $L\in\wt M$, the tangent space of $\wt M$ at $L$ is identified as $T_{L}\wt M \cong L^* \otimes L^\perp/L$.
On the other hand, for $L'=\pi(L)\in M$, the tangent space of $M$ at $L'$ is a subset $T_L'M\subset L'^*\otimes\V/L'$ specified in Lemma \ref{lem:TOM}.
Now, let $L=\<v_++v_-\>$ and let $w\in T_{L}\wt M$ be interpreted as a linear map $L\to L^\perp/L$ given by
\begin{align}
v_++v_- \mapsto (w_++w_-) / \<v_++v_-\> ,
\label{eq-w}
\end{align}
where $w_\pm\in L^\perp\cap\V_\pm$. 
Then $L'=\<v_+,v_-\>$ and the projection $w'=T\pi(w) \in T_L'M$, interpreted as a linear map $L'\to\V/L'$, is given as
\begin{align}
v_+ \mapsto w_+ / \<v_+,v_-\>, \quad v_- \mapsto w_- / \<v_+,v_-\> ,
\label{eq-w'}
\end{align}
Indeed, this is a para-complex map and the assumption $w_++w_-\in\<v_++v_-\>^\perp$ implies that the composition with $\V/L'\to\V/L'^\perp\cong L'^*$ is skew.
In these terms, the distinguished vector field \eqref{eq-K} spanning the vertical subbundle of the projection $\pi:\wt M\to M$ is as follows:
for $L=\<v_++v_-\>$ as above, the vector $\k_L\in T_L\wt M$ corresponds to the linear map $L\to L^\perp/L$ given by
\begin{align}
v_++v_- \mapsto (v_+-v_-) / \<v_++v_-\> .
\label{eq-KL}
\end{align}

From the rough description of the construction in Section \ref{LC-feff-rough}, we already know that the subbundle $\k^\perp\subset T\wt M$ projects to $\D\subset TM$.
In the current terms, the discussion can be refined as follows.
For any $L\in\wt M$ and $w\in T_L\wt M$, let  $W=\op{im}(w)+L$ be the auxiliary vector subspace in $\V$.
Let us denote the projected counterparts as $L'=\pi(L)\in M$, $w'=T\pi(w)\in T_LM$  and $W'=\op{im}(w')+L'\subset\V$.
Then, clearly, $W\subset W'$ and $W'=\K(W)$.
Further, it follows that $W\perp\k$ if and only if $W'$ is degenerate.
In addition, $W\perp\k$ and $W$ is null if and only if $W'$ is null.
This, together with Lemma \ref{lem:TOM}, provides an alternative control over the types of tangent vectors under the present correspondence.

Passing to the curves, we recall that \emph{null-geodesics} of a conformal parabolic geometry are the canonical curves of type $S\subset\wt\g_{-1}$, where $\wt\g_{-1}$ is the negative part of the conformal grading of $\wt\g=\mf{so}(n+2,n+2)$ and $S$ is the subset of null-vectors with respect to the standard inner product.
Null-geodesics of a conformal structure on $\wt{M}$, i.e. null-geodesics of the corresponding normal parabolic geometry, are the true geodesics of any metric from the conformal class.
In particular, as unparametrized curves, they are uniquely given by a tangent direction in one point.
In the homogeneous model they are simply the null straight lines, i.e. the real projectivizations of real null-planes $W\subset\V=\R^{n+2,n+2}$.
The projections of null-geodesics are thus contained in the para-complex hulls $W'=\K(W)$.
From the previous specification and Propositions \ref{prop:model-chain} and \ref{prop:model-null-chain}, we know that these are the subspaces in which chains and null-chains are contained, respectively.
If $W\not\perp\k$ then $W'$ is non-degenerate and its intersection with $\N$ determines a unique chain.
If $W\perp\k$ then $W'$ is null and there are many potentially matching null-chains.
A closer look leads to the following result.

\begin{prop}
\label{prop:fefferman correspondence}
Let $(\wt\G\to\wt{M},\wt\om)$ be the conformal parabolic geometry induced from a LC structure on $M$ by the Fefferman-type construction and let $\k$ be the null-vector field spanning  the vertical subbundle of $\wt{M}\to M$.
Then:
\begin{enumerate}[(a)]
\item Flow lines of $\k$ are null-geodesics in $\wt{M}$ that project to points in $M$.
\item Null-geodesics in $\wt{M}$ that are not perpendicular to $\k$ project to chains in $M$ and all chains are of this form.
\item Null-geodesics in $\wt{M}$ that are both perpendicular and transverse to $\k$ project to null-chains in $M$ and all null-chains are of this form.
\end{enumerate}
\end{prop}

For the sake of simpler presentation, we do not distinguish generic null-chains and null-chains of type $E$, respectively $F$.
In the spirit of Proposition \ref{prop:model-null-chain}, the latter curves can always be seen as degenerate cases of the former ones, where the generic null-directions from $\D$ are substituted by those from $E$, respectively $F$.
The corresponding preimages in $T\wt{M}$ are obtained by a substitution of generic null-directions from $\k^\perp$ by those from $\wt{E}$, respectively $\wt{F}$; cf. the description after \eqref{eq-K}.

\begin{proof}
The general approach is as follows.
For a tangent direction at a point in $M$, we consider all its null-lifts to $\wt{M}$, all the corresponding null-geodesics and their projections back to $M$.
Then we show that the projected curves are chains, respectively null-chains, tangent to the initial direction.
The case (a) is, of course, exceptional.

It is enough to analyze this correspondence in the homogeneous model, which translates to the general case via the notion of development of curves.
We use the ambient vector descriptions as above.
Consistently with our previous notation, the origin in $M$ is the para-complex null-line $O=\<e_0,e^{n+1}\>$ in $\V$.
The corresponding fiber in $\wt M$ consists of real null-lines in $O$ that are not contained in $\V_\pm$; it can be described as 
\begin{align*}
\pi^{-1}(O)=\{ \<e_0+l e^{n+1}\>, \text{ where } l\in\R\setminus\{0\} \} .
\end{align*}

(a)
Clearly, the fiber $\pi^{-1}(O)$ is a null-geodesic in $\wt{M}$ and the vector field $\k$ is everywhere tangent to it, cf. the description in \eqref{eq-KL}.

(b) 
Let us take the tangent vector at $O$ of the model chain \eqref{eq-chain-a} and let  $L=\<e_0+l e^{n+1}\>$ be an arbitrary element from $\pi^{-1}(O)$.
This tangent vector lifts to a unique null-vector in $T_L\wt M$, namely, 
\begin{align}
e_0+l e^{n+1} \mapsto (e_{n+1}-l e^0) / \<e_0+l e^{n+1}\> ,
\label{eq-null-lift}
\end{align}
cf. the description in \eqref{eq-w}.
The null-geodesic through $L$ in the direction of \eqref{eq-null-lift} can be parametrized by $t\in\R$ as
\begin{align}
\< e_0+l e^{n+1} + t(e_{n+1}-l e^0) \>
=\< (e_0 + t e_{n+1}) +l (-t e^0 + e^{n+1}) \> .
\label{eq-null-geod}
\end{align}
For any $l\in\R$, i.e. for any initial point $L\in\pi^{-1}(O)$, 
the null-geodesic \eqref{eq-null-geod} projects to the chain \eqref{eq-chain-a}.
Hence the statement follows.

(c)
Let us take the tangent vector at $O$ of the model null-chain \eqref{eq-nchain-a} and let  $L=\<e_0+l e^{n+1}\>$ be as above.
Contrary to the previous case, the tangent vector does not lift uniquely to a null-vector in $T_L\wt M$; 
all such lifts are controlled by the parameter $k\in\R$ so that
\begin{align}
e_0+l e^{n+1} \mapsto \left( (e_1-le^n)+k(e_0-le^{n+1}) \right) / \<e_0+l e^{n+1}\> .
\label{eq-nulll-lift}
\end{align}
The null-geodesic through $L$ in the direction of \eqref{eq-nulll-lift} can be parametrized by $t\in\R$ as
\begin{align}
\< e_0+l e^{n+1} + t( (e_1-le^n)+k(e_0-le^{n+1}) ) \> 
= \< (1+kt)e_0+te_1 + l( -te^n+(1-kt)e^{n+1} ) \> .
\label{eq-nulll-geod}
\end{align}
These null-geodesics project to the family of curves
\begin{align}
\< (1+kt)e_0+te_1, -te^n+(1-kt)e^{n+1} \> ,
\label{eq-null-proj}
\end{align}
depending on $k\in\R$, but independent of $l\in\R$.
Comparing with \eqref{eq-nnchain-a}, curves \eqref{eq-null-proj} form the 1-parameter family of null-chains passing through the origin in the same direction as \eqref{eq-nchain-a}.
Hence the statement follows.
\end{proof}

Finally, for an integrable LC structure, our Fefferman-type construction preserves the normality condition, see Proposition \ref{prop-feff-normal}.
In such case, the null-geodesics of the induced conformal parabolic geometry are just null-geodesics of the underlying conformal structure.
Putting things together, we conclude with the following result.

\begin{theorem} \label{cor:integrable chains}
Let the conformal structure on $\wt{M}$ be induced from an integrable LC structure on $M$ by the Fefferman-type construction, and let $\k$ be the null conformal Killing field spanning the vertical subbundle of $\wt{M}\to M$.
Then chains, respectively null-chains, in $M$ are exactly the projections of null-geodesics in $\wt{M}$ that are not, respectively are, perpendicular to $\k$.
\end{theorem}

\subsection{Chains as Kropina geodesics} \label{sec-Kropina}

To have the correspondence from Theorem \ref{cor:integrable chains}, we assume the LC structure on $M$ is integrable.
Associated to the Fefferman metric on $\wt{M}$, we have Kropina metrics on $M$, which are metrics of pseudo-Finsler type defined off the contact distribution $\D\subset TM$.
Hence we get the description of chains as geodesics of a Kropina metric, which also yields further consequences for the geometry of chains.
This part is analogous to the CR case studied in \cite{Cheng2019}.

Let $\wt{g}$ be a representative Fefferman metric and $\k$ be the null conformal Killing field spanning the vertical subbundle of the projection $\pi:\wt{M}\to M$ as above.
Then, for any local section $\psi:U\subset M\rightarrow \widetilde{M}$ , the \emph{Kropina metric} on $U$ associated to $\wt{g}$ and $\psi$ is defined by 
\begin{align}\label{eq-Kropina}
F_{\psi}(\xi) := \frac{\wt{g}(\psi_*\xi,\psi_*\xi)}{\wt{g}(\k,\psi_*\xi)} ,
\end{align}
where $\xi\in TM\setminus\D$ and $\psi_*$ denotes the tangent map of $\psi$.
Indeed, this is well-defined only off the contact distribution, since $\k^\perp\subset T\wt{M}$ corresponds to $\D\subset TM$ under the projection.
The definition is independent of the choice of a particular metric from the conformal class of $\wt{g}$, but depends on the choice of section $\psi$.
However, a change of the local section only leads to a change of the Kropina metric by an exact 1-form, cf.  \cite[Section 2]{Cheng2019}. 
More precisely, let  $(s,x^i,u,p_i)$ be a preferred local coordinate on $\wt M$ as given in Theorem \ref{prop-Feff-integrable}, let $\psi':U\to\wt M$ be another section and let the section change be measured by the function $f$ defined on $U$ by $f=s\circ\psi'-s\circ\psi$. 
By Remark \ref{rem-feff}(ii), the vector field $\mathcal{K}=\partial_s$ is a null true Killing field for a representative metric $\wt{g}$. 
The vectors $\psi'_*\xi$ and $\psi_*\xi+df(\xi)\partial_s$ have the same norm with respect to $\wt{g}$, so putting it all together, the induced Kropina metrics are related by 
\begin{align*}
F_{\psi'}(\xi) 
=F_{\psi}(\xi) + \frac{2\wt{g}(\psi_*\xi,df(\xi)\partial_s)}{\wt{g}(\partial_s,\psi_*\xi)} 
=F_{\psi}(\xi) + 2df(\xi) .
\end{align*}


Similar to Riemannian geometry, the unparametrized geodesics of a Kropina metric $F$ are the stationary curves of the following $F$-length functional among all curves with the same endpoints:
\begin{align}
L(\gamma):=\int_a^b F(\gamma(t),\dot{\gamma}(t)) \d t,
\label{eq-L}
\end{align}
where $\gamma: [a,b]\rightarrow M$. Since Kropina metrics have the correct homogeneity degree with respect to $\dot\ga$, the functional \eqref{eq-L} is invariant under orientation preserving reparametrizations.
Unlike Riemannian geodesics, the geodesics of a Kropina metric do not have arbitrary directions but they are everywhere transverse to the distribution $\D\subset TM$.
Since Kropina metrics corresponding to different sections of the projection $\pi$ differ by an exact 1-form, their geodesics are the same.

The Euler--Lagrange equations corresponding to \eqref{eq-L} are given by
\begin{align}
\dfrac{\partial F}{\partial \gamma^{\alpha}} - \dfrac{\d}{\d t}\dfrac{\partial F}{\partial \dot{\gamma}^{\alpha}} = 0,
\label{eq-EL}
\end{align}
where the superscript $\al=1,\dots,n+1$ denotes coordinate components on $M$.
Overall, Kropina geodesics are solutions to the ODE system \eqref{eq-EL}, which are everywhere transverse to the contact distribution $\D$.

There is a local correspondence between the geodesics of the Kropina metric $F$ and the null-geodesics of the Fefferman metric $\wt{g}$ that are not perpendicular to $\k$.
This is presented a generalization of the so-called Fermat's principle in \cite[Theorem 2.1]{Cheng2019}, cf. also \cite[Theorem 7.8]{Erasmo}.
In this part, the fact that $\k$ is a null Killing field of $\wt{g}$ plays an important role.
Combining the current correspondence with Theorem \ref{cor:integrable chains}, we obtain the following result.

\begin{theorem} \label{prop-Kropina}
Let the conformal structure on $\wt{M}$ be induced from an integrable LC structure on $M$ by the Fefferman-type construction, and let $F$ be a locally defined Kropina metric as in \eqref{eq-Kropina}.
Then the chains in $M$ are precisely the geodesics of the Kropina metric $F$.
\end{theorem}

This result provides a tool for deriving the system of ODEs for chains, which appears to be very efficient compared to other strategies.
An instance, for 3-dimensional LC structures, is presented in Section \ref{curves-dim3}.
Other applications of the current approach to chains follow.
They concern known results on the  determinacy of the LC structure by its chains or on the character of chains themselves, but with weakened assumptions: 
instead of considering the whole family chains, we are allowed to restrict only to a ``sufficiently big'' subset.

A family of (unparametrized) curves on $M$ is \emph{sufficiently big} if, for any $x\in M$, the set of tangent vectors at $x$ for these curves contains a nonempty open subset of $T_xM$. 
Two Kropina metrics are said to be \emph{projectively equivalent} if their geodesics coincide on a sufficiently big family of curves. 
By \cite[Theorem 1.2]{Cheng2019}, projectively equivalent Kropina metrics of the form $F_1=\frac{g_1}{\si_1}$ and $F_2=\frac{g_2}{\si_2}$ satisfy 
\begin{align}
\ker\si_1=\ker\si_2 \quad\text{and}\quad F_2=c F_1+\be, 
\label{eq-F1F2}
\end{align}
where $c$ is a constant and $\beta$ is a closed 1-form,
provided that $\ker\si_i$ are non-integrable distributions and $g_i$ are non-degenerate on $\ker\si_i$.
In particular, this holds for Kropina metrics \eqref{eq-Kropina} induced by integrable LC structures.
This is one of the key inputs for the following statement.

\begin{prop}
\label{prop:chain fix LC}
Two integrable LC structures share a sufficiently big family of chains if and only if the LC structures coincide.
\end{prop}

\begin{proof}
For the only non-trivial implication in the equivalence, 
let $\mathcal{D}_1=E_1\oplus F_1$ and $\mathcal{D}_2=E_2\oplus F_2$ be two integrable LC structures on $M$ with a same sufficiently big family of chains.
By this assumption and Theorem \ref{prop-Kropina}, there are locally defined Kropina metrics $F_1$ and $F_2$ that are projectively equivalent.
Thus, by \eqref{eq-F1F2}, the contact distributions $\D_1$ and $\D_2$ agree and Kropina geodesics of $F_1$ and $F_2$ coincide in general, not only on a sufficiently big set.
Then, by \cite[Corollary 4.4]{vcap2009}, it follows that the two LC structures coincide.
\end{proof}

Focusing on the chains themselves, they form a far more complicated system than geodesics of an affine connection, though they are both given by initial conditions of the same order.
More precisely, there is no affine connection which has all chains as its geodesics, see \cite[Theorem 4.1]{vcap2009}.
Again, with our current description in Theorem \ref{prop-Kropina}, we generalize this result as follows.

\begin{prop}
\label{prop:chain not affine}
There is no affine connection which has any sufficiently big family of chains among its geodesics.
\end{prop}

\begin{proof}
On the one hand, geodesics of a pseudo-Finsler metric are given by the equation
\begin{align}
\label{eqn:Finsler geodesic}
\ddot{x}^k =-\Gamma^k_{ij}(x,\dot{x})\dot{x}^i\dot{x}^j ,
\end{align}
where the coefficients $\Gamma^k_{ij}(x,\dot{x})$ are formal Christoffel symbols given by algebraic formulas of the first-order derivatives of the fundamental tensor of that metric, cf. \cite[Section 5.3]{BCS}. 
(The local coordinates here have nothing to do with the adapted coordinates we use elsewhere.)
Chains are geodesics of a Kropina metric whose special form \eqref{eq-Kropina} implies that, for each fixed $x$, the right hand side of \eqref{eqn:Finsler geodesic} is a rational function of $\dot{x}^i$ of constant coefficients.

On the other hand, geodesics of an affine connection are given by an analogous equation as in \eqref{eqn:Finsler geodesic}, but with the coefficients depending  only on $x$.
Allowing reparametrizations, we deal with the equation 
\begin{align}
\label{eqn:std-geodesic}
\ddot{x}^k =-\wh\Gamma^k_{ij}(x)\dot{x}^i\dot{x}^j +f(x)\dot{x}^k ,
\end{align}
for some function $f$, where $\wh\Gamma^k_{ij}$ are the standard Christoffel symbols of the affine connection.
Putting \eqref{eqn:Finsler geodesic} and \eqref{eqn:std-geodesic} together, we obtain, for each fixed $x$, a system of purely algebraic equations on $\dot{x}$ with constant coefficients. 
Thus, if a sufficiently big family of Kropina geodesics is among geodesics of an affine connection, then all Kropina geodesics are affine geodesics.
But, by \cite[Corollary 3.6]{Cheng2019}, this is impossible, hence the statement follows.
\end{proof}

\section{LC structures induced by projective} \label{LC-by-projective}

There is a distinguished class of LC structures that are induced by the projective ones.
This relation is clearly visible in the model description in Section \ref{sec-LC-model}. It is studied in full generality firstly in \cite{Takeuchi1994}.
Also, this correspondence allows a Cartan-geometric interpretation that fits nicely to our present approach.
Although LC structures arising this way are typically not integrable, it is still possible to treat some of topics discussed above in an adequate detail.
In particular, we present an explicit description of the Fefferman metric and the correspondence of canonical curves.
As an interesting side result, we obtain a characterization of such LC structures in terms of its defining functions.

\subsection{The correspondence} \label{LC-proj-prelim}

For a smooth manifold $\un{M}$, there is a canonical contact structure on the projectivized cotangent bundle $M:=\Proj T^*\un{M}$ of $\un{M}$.
From now on, we assume $\dim\un{M}=n+1$ and so $\dim M=2n+1$.
A projective structure on $\un{M}$, which is given by a family of torsion-free affine connections with the same unparametrized geodesics, gives rise to a LC structure on $M$ as follows. First note that the kernel of the tautological 1-form on $T^*\underline{M}$ descends to a contact distribution $\mathcal{D}$ on $M$.
Each connection from the projective class determines a horizontal distribution in $TT^*\un{M}$ complementary to the vertical distribution of the projection $T^*\un{M}\to\un{M}$.
This decomposition depends on the connection, however, the vertical distribution and the common intersection of horizontal distributions of connections within the projective class descend to a Legendrian decomposition of the contact distribution $\D\subset TM$, i.e. a half-integrable LC structure on $M$.
It follows that the LC structure is flat if and only if the initial projective structure is flat, cf. \cite{Takeuchi1994}.

Both projective and LC structures are parabolic geometric structures, whose model correspondence is described in Section \ref{sec-LC-model}.
In particular, let $G=SL(n+2,\R)$ be the principal Lie group and $P\subset\un{P}$ be the nested parabolic subgroups of $G$ as described  in Section \ref{sec-LC-model}.
Denoting the grading of the Lie algebra $\g=\mf{sl}(n+2,\R)$ corresponding to $\un{P}\subset G$ as $\g=\un\g_{-1}\oplus\un\g_0\oplus\un\g_1$, it is related to the one in \eqref{eq-LC-g} so that 
\begin{align*}
\un\g_{-1} = \g_{-2}\oplus\g_{-1}^E, \quad
\un\g_0 = \g_{-1}^F\oplus\g_0\oplus\g_1^F,\quad
\un\g_1 = \g_1^E\oplus\g_2.
\end{align*}
Again, the model fibration 
\begin{align*}
\xymatrix@R=1.4\baselineskip{
& G \ar[d] \ar@/_4mm/[ddl] \\
&  M\cong G/P \ar[dl] \\
\un{M}\cong G/\un{P} & \\
}
\end{align*}
translates directly to general curved settings, where the group $G$ with its Maurer--Cartan form is substituted by a bundle $\G$ with a Cartan connection $\om$.
Let $(\G\to\un{M},\om)$ be the normal parabolic geometry of type $(G,\un{P})$ associated to a projective structure on $\un{M}$. The correspondence space $M:=\G/P$ is identified with $\Proj T^*\un{M}$ and $(\G\to M,\om)$ is the normal parabolic geometry of type $(G,P)$ associated to the induced LC structure.
The vertical subbundle of the projection $M\to\un{M}$ is the Legendrian subdistribution $F$ of the contact distribution $\D\subset TM$.
From the concrete description of the LC harmonic curvature (and its horizontality), it follows that there is just one potentially nontrivial component which is fully determined by its projective counterpart. 
This refines significantly the above mentioned characterization of the flatness of the induced LC structure.
Note that, for $n>1$, the flatness is a necessary (not only sufficient) condition of the integrability of the induced LC structure.
Moreover, vanishing of the other harmonic curvature components locally characterizes the LC structures arising this way, 
see \cite{Cap2005} or \cite[Section 4.4.2]{Cap2009}.

Composing the current setup with the Fefferman-type construction from Section \ref{LC-feff-rough} 
yields a construction relating the projective structure on $\un{M}$ and a conformal structure on $\wt{M}$.
Proposition \ref{prop-feff-normal} then readily extends as follows:

\begin{prop}[{\cite[Proposition 3.8]{Hammerl2017}}] \label{prop-LC-proj-feff} 
Let $(\G\to\un{M},\om)$ be the normal parabolic geometry for an underlying projective structure on $\un{M}$. 
Let $(\G\to M,\om)$ be the intermediate LC parabolic geometry 
and let $(\wt\G\to\wt{M},\wt\om)$ be the conformal parabolic geometry obtained by the Fefferman construction.
Then $\wt\om$ is normal if and only if $\dim\un{M}=2$ or $\om$ is flat.
\end{prop}

\subsection{Explicit Fefferman metric} \label{LC-proj-feff}

Also for LC structures induced by general projective structures, there is an explicit coordinate description of the induced conformal Fefferman metric, expressed in terms of Christoffel symbols of a representative affine connection from the projective class.
In this context, we highlight that the Fefferman space $\wt{M}$ is identified with an appropriately weighted cotangent bundle of $\un{M}$ and the so-called \emph{Patterson--Walker metrics} are directly detected in the conformal class, see \cite[Section 6.1]{Hammerl2017}.
Assuming that a representative affine connection from the projective class is special, i.e. preserves a volume form on $\un{M}$, we identify $\wt{M}$ with $T^*\un{M}$. Then the corresponding Patterson--Walker metric on $\wt{M}$ is given by a natural pairing between the vertical and the horizontal distribution.
For an explicit description, let $(x^a)$ be a local coordinate on $\un{M}$ and $(y_b)$ be the canonical fiber coordinates on $T^*\un{M}$, i.e. so that the tautological 1-form has the form $y_a\d x^a$.
Here and below, the indices $a,b$ etc. run from 1 to $n+1$ and repeated indices indicate the sum.
For $\Ga^c_{ab}$ being the Christoffel symbols of a special torsion-free affine connection, the corresponding Patterson--Walker metric has the form 
\begin{align}
g = \d x^a \odot \d y_a - y_c\, \Ga^c_{ab} \d x^a \odot \d x^b .
\label{eq-Feff-PW}
\end{align}
Indeed, for a particular value of the weight (which is $w=2$ according to standard conventions), it follows that the projective change of connection,
\begin{align}
\wh\Ga^c_{ab} = \Ga^c_{ab}+\de^c_a\Up_b+\de^c_b\Up_a, 
\label{eq-proj-change}
\end{align}
where $\Up=\d f$ and $f$ is a nowhere vanishing function on $\un{M}$,
induces a conformal change of the metric, $\wh g=e^{wf} g$, see \cite[Section 3.3]{Hammerl2019}.

Besides the previous expression, we also examine the two-step process indicated in Proposition \ref{prop-LC-proj-feff}.
Let $(\G\to\un{M},\om)$ be the normal parabolic geometry associated to the projective structure on $\un{M}$ and let $\un\phi:\un{M}\to\G$ be a local section of the bundle projection.
The components of the corresponding gauge $\un\phi^*\om:T\un{M}\to\g$ are denoted by
\begin{align}
\begin{pmatrix}
\un\om^0_0 & \un\om^0_b  \\
\un\om^a_0 & \un\om^a_b 
\end{pmatrix} ,
\label{eq-gauge1}
\end{align}
which reflects the block decomposition of $\g=\mf{sl}(n+2,\R)$ corresponding to projective geometries discussed in Section \ref{LC-proj-prelim}.
One may always choose a local section so that 
\begin{align}
\un\om^a_0=\d x^a, \quad
\un\om^a_b=\Ga^a_{bc}\d x^c, \quad
\un\om^0_0=0 ,
\label{eq-gauge2}
\end{align}
where $(x^a)$ are local coordinates on $\un{M}$ and $\Ga^a_{bc}$ are Christoffel symbols of a special torsion-free affine connection from the projective class as above.
This, in particular, means that $\sum\limits_{a=1}^{n+1} \Ga^a_{ac}=0$.
In fact, the normality of $\om$ forces this setup to be completed by
$\un\om^0_b=-\frac1{n-1}R_{bc}\d x^c$, 
where $R_{ab}$ is the Ricci tensor of the representative connection, see e.g. \cite{Crampin2007}.

The typical fiber of the projection $M\to\un{M}$, which is $P/\un{P}$ by construction, is identified with the nilpotent subgroup $\exp\g_{-1}^F\subset P$.
We use the parametrization
\begin{align}
\exp 
\begin{pmatrix}
0 & 0 & 0 \\
0 & 0 & 0 \\
0 & p_j & 0
\end{pmatrix} ,
\label{eq-expF}
\end{align}
where we refer to the block decomposition \eqref{eq-LC-g} and  $\g_{-1}^F\cong\R^n$.
Composing the gauge section $\un\phi:\un{M}\to\G$ with the bundle projection $\G\to M$ gives a preferred section of the fibration $M\to\un{M}$.
Any element of $M$ is related to the image of this section by the action of \eqref{eq-expF}, which induces the fiber coordinates $(p_j)$.
It follows that, under the substitution $x^{n+1}=u$, this setting is compatible with the one from Section \ref{sec-LC-general}, which yields the following.

\begin{prop} \label{LC-proj-char}
Let $\D=E\oplus F$ be a half-integrable LC structure with integrable $F$-leaves on a manifold $M$ and let 
\begin{align}
\si=\d u -p_i \d x^i, \quad
\th^i=\d x^i, \quad
\pi_i=\d p_i -f_{ij}\d x^j
\label{eq-cofram}
\end{align}
be an adapted coframe, in local coordinates $(x^i,u,p_i)$, such that $E=\ker\<\si,\pi_i\>$ and $F=\ker\<\si,\th^i\>$.
Let $\un{M}$ be the local leaf space of the integrable distribution $F$, with coordinates $(x^i,u)$.
The LC structure on $M$ is induced by a projective structure on $\un{M}$ if and only if the functions $f_{ij}$ are cubic polynomials in $p_i$ of the form
\begin{align}
\begin{split}
f_{ij} 
= -\Ga^{n+1}_{i\,j} 
&+ p_k \Ga^k_{i\,j} - p_j\Ga^{n+1}_{n+1\,i} - p_i \Ga^{n+1}_{n+1\,j} + \\
&+ p_i p_k \Ga^k_{n+1\,j} + p_j p_k \Ga^k_{n+1\,i} - p_i p_j \Ga^{n+1}_{n+1\,n+1}
+ p_i p_j p_k \Ga^k_{n+1\,n+1} ,
\end{split}
\label{eq-fij}
\end{align}
where the functions $\Gamma^c_{ab}$, 
being the Christoffel symbols of a representative connection from the projective class,
depend on $(x^i,u)$ and satisfy $\Ga^c_{ab}=\Gamma^c_{ba}$ and $\sum\limits_{a=1}^{n+1} \Ga^a_{ac}=0$.
\end{prop} 

\begin{proof}
For a projective structure represented by $\Ga^c_{ab}$, let us consider the gauge \eqref{eq-gauge1} on $\un{M}$ with components as in \eqref{eq-gauge2}.
The components changes along the fibration $M\to\un{M}$ according to the transformation formula \eqref{eq-gauge-change} where we substitute \eqref{eq-expF}.
Focusing on the $\g_-$-part yields the adapted coframe \eqref{eq-gauge0} with $f_{ij}$ as in \eqref{eq-fij}.
Indeed, this expression is symmetric in $i$ and $j$ and invariant under the projective change of connection \eqref{eq-proj-change}.

Conversely, let a half-integrable LC structure be given by the functions $f_{ij}$ of the form \eqref{eq-fij} with $\Ga^c_{ab}$ depending only on $(x^i,u)$ and satisfying $\Ga^{c}_{ab}=\Ga^c_{ba}$ and $\sum\limits_{a=1}^{n+1}\Ga^a_{ac}=0$.
The coefficients of the polynomial form  a linear system of equations on $\Ga^c_{ab}$ which is solvable and fixes all $\Ga^c_{ab}$.
Interpreting $\Ga^c_{ab}$ as Christoffel symbols of an affine connection on $\un{M}$, we have a representative connection of the projective class, whose associated LC structure is the one we started with.
\end{proof}

\begin{remark} 
In the case when $n=1$, i.e. $\dim M=3$, the right-hand side of \eqref{eq-fij} becomes a cubic polynomial in one variable, so the associated PDE system $\eqref{eqn: LC PDE}$ becomes a single ODE of second order. In this case, Proposition \ref{LC-proj-char} becomes the famous characterization of geodesic equations by Cartan.
More details on this special situation are in Section \ref{sec-dim3}.
\end{remark}

Now, we extend the current LC structure by the Fefferman-type construction and introduce the fiber coordinate on the Fefferman space as in Section \ref{LC-feff-explicit}. It leads us to the following coordinate description of the Fefferman metric.

\begin{prop} \label{prop-Feff-proj}
Let a half-integrable LC structure on $M$ be induced by a projective structure on $\un{M}$.
Let $(x^a)$ be a local coordinate on $\un{M}$, $\Ga^c_{ab}$ be the Christoffel symbols of an affine connection from the projective class and $(x^i,u,p_i)$ be a local coordinates on $M$ as in Proposition \ref{LC-proj-char}, where $u=x^{n+1}$.
Let $[\wt{g}]$ be the induced conformal structure on the Fefferman space $\wt{M}$ and let $s$ be the induced fiber coordinate of the projection $\wt{M}\to M$.
Then a representative metric from the conformal class has the form
\begin{align}
\wt g = (\Gamma^{n+1}_{bc}-p_k\Gamma^k_{bc}) \d x^b \odot \d x^c +\d x^i \odot \d p_i +2(\d x^{n+1}-p_i\d x^i) \odot \d s .
\label{eq-Feff-proj}
\end{align}
Moreover, the formula above is related to the description in \eqref{eq-Feff-PW} by the coordinate transformation 
\begin{align}
p_i=-\frac{y_i}{y_{n+1}}, \quad
s=-\frac12\log|y_{n+1}| .
\label{eq-transf}
\end{align}
\end{prop}

\begin{proof}
The rough form of the Fefferman metric is the same as in \eqref{eq-Feff-integrable}.
Indeed, this is derived from \eqref{eq-feff-form} where the integrability does not enter.
From the same expansion as in the proof of Proposition \ref{LC-proj-char}, one reads the remaining term we need as
\begin{align}
\varpi =\om^0_0+\om^{n+1}_{n+1} 
= -(\Ga^{n+1}_{n+1\,c} -p_i \Ga^i_{n+1\,c}) \d x^c.
\label{eq-varpi-proj}
\end{align}

As at the end of proof of Theorem \ref{prop-Feff-integrable}, it easily follows that the previous expression transforms along the fibration $\wt M\to M$ as $\varpi\mapsto\varpi+2\d s$.
Including the substitutions \eqref{eq-cofram}, the formula \eqref{eq-Feff-proj} follows after an easy manipulation.
The relation to \eqref{eq-Feff-PW} via the coordinate transformation \eqref{eq-transf} is a direct computation.
\end{proof}

\subsection{Correspondence of curves} \label{LC-proj-curves}

Concerning the correspondence of canonical curves, we give comments on both steps of the previous construction.
Firstly, we recall that the geodesics of a projective parabolic geometry are canonical curves of type $S=\un\g_{-1}$, which is the negative part of the corresponding grading of $\g=\mf{sl}(n+2,\R)$ as in Section \ref{LC-proj-prelim}.
Geodesics of a projective structure on $\un{M}$ are, of course, the true geodesics of any affine connection from the projective class.
In the homogeneous model, projective geodesics are just straight lines, the 1-dimensional subspaces of the real projective space.

For the correspondence between the geodesics on $\un{M}$ and chains, respectively null-chains, of the induced LC structure on $M=\Proj T^*\un{M}$, we refer to the generalities in Sections \ref{curves-prelim} and \ref{LC-proj-prelim} and the model observations in Propositions \ref{prop:model-chain}, respectively \ref{prop:model-null-chain}.
 Part (b) of these propositions provides a link to the underlying projective geometry.
In the notation used there, the projection $M\to\un{M}$ is given by $(p,H)\mapsto p$.
Altogether, we easily conclude with

\begin{prop} \label{prop-curves-proj}
Let the LC structure on $M$ be induced by a projective structure on $\un{M}$ and let $F\subset TM$ be the vertical distribution of the projection $M\to\un{M}$.
Then both chains and null-chains in $M$ project to geodesics in $\un{M}$ and all geodesics on $\un{M}$ are of this form.
\end{prop}

Note that null-chains of type $F$ project to points.

Passing to the Fefferman space, we are concerned about the correspondence between conformal null-geodesics in $\wt{M}$ and LC chains, respectively null-chains, in $M$ as in Section \ref{curves-correspondence}.
We still have the general description from Proposition \ref{prop:fefferman correspondence}.
Spotting the correspondence of curves for underlying geometric structures, a quick answer is bounded by the normality issues summarized in Proposition \ref{prop-LC-proj-feff}.
Hence, for LC structures induced by projective, the correspondence  described in Theorems \ref{cor:integrable chains} and \ref{prop-Kropina} holds in the lowest dimensional or in the flat case.

However, even in general case, we have some control over the correspondence of curves due to the specificity of the induced conformal Cartan connection.
Although it is generally not normal, its origin guarantees it is ``half-normal'' which allows a satisfactory control over the normalization process, cf. \cite{Hammerl2017}.

\begin{prop} 
Let the LC structure on $M$ be induced by a projective structure.
Let $\wt{M}$ be the induced Fefferman conformal manifold, let $\k$ be the null conformal Killing field spanning the vertical distribution of $\wt{M}\to M$ and let $\wt{W}$ be the conformal Weyl curvature.
Then a conformal null-geodesic $c:I\to\wt{M}$ projects to a chain or a null-chain in $M$ if and only if
$$
\wt{W}\left(\k,c'(t)\right)\left(c'(t)\right) =0, 
\quad \text{for all $t\in I$} , 
$$ 
where $c'(t)=\frac{d}{dt}c(t)$ and $\wt{W}$ is seen as a 2-form on $\wt{M}$ with values in $\op{End}(T\wt{M})$.
\end{prop}

\begin{proof}
Let $\wt{\om}^{ind}$ and $\wt\om^{nor}$ be the induced and the normalized conformal Cartan connection, respectively, for the Fefferman-type construction.
Their difference is interpreted as a 1-form on $\wt{M}$ with values in the Lie algebra $\wt\g$.
In our present situation, this 1-form is given by the contraction of $\k$ and the normal Cartan curvature, see \cite[Theorem 5.7]{Hammerl2017}.
Let $\wt\nabla^{ind}$ and $\wt\nabla^{ind}$ be the related Weyl connections, i.e. compatible affine connections corresponding to the same scale.
Their difference then corresponds to a 1-form on $\wt{M}$ with values in $\op{End}(T\wt{M})$ which is given by the contraction of $\k$ and the Weyl curvature $\wt{W}$.
I.e., for any $\xi\in T\wt{M}$, the difference $\wt{\nabla}^{ind}_\xi-\wt{\nabla}^{nor}_\xi$ equals to $\wt{W}(\k,\xi)$ up to a nonzero multiple.

Null-geodesics of the induced and the normalized Cartan geometry on $\wt{M}$ are the geodesics of the affine connection $\wt{\nabla}^{ind}$ and $\wt{\nabla}^{nor}$, respectively.
According to the previous comparison, an unparametrized geodesic $c$ of $\wt\nabla^{ind}$ is a geodesic of $\wt\nabla^{nor}$ if and only if $\wt{W}(\k,c')(c')$ is proportional to $c'$. 
By the symmetries of the Weyl tensor, the latter condition means that $\wt{W}(\k,c')(c')=0$.
Projections of null-geodesics of the induced connection are understood in Proposition \ref{prop:fefferman correspondence}, hence the statement follows.
\end{proof}

Let us recall that the tested condition is automatically satisfied in the lowest dimensional or in the flat case.
The types of target curves are controlled by relations of source curves to the vector field $\k$ as in Proposition  \ref{prop:fefferman correspondence}.

\section{LC structures in dimension three} \label{sec-dim3}

LC structures on 3-dimensional manifolds, which correspond to $n=1$ in our previous notation, are indeed special.
The contact distribution $\D\subset TM$ has rank 2 and the components of the Legendrian decomposition $E\oplus F=\D$ have rank 1.
In particular, 3-dimensional LC structures are automatically integrable.
We use the adapted local coordinates as in Section \ref{sec-LC-general}, although we write $y$ instead of $u$.
The decomposition \eqref{eq-frame} then takes the form
$E = \left\< \parder{}{x} +p\parder{}{y} +f\parder{}{p} \right\>$ and
$F = \left\< \parder{}{p} \right\>$,
the adapted coframe \eqref{eq-coframe} is 
$\si=\d y -p \d x$, $\th=\d x$, $\pi=\d p -f\d x$,
where  $f=f(x,y,p)$ is the defining function.

\subsection{Subordinate path geometry}

Forming a local leaf space $\un{M}$ of the distribution $F$ (or, dually, of $E$), $M$ is identified with $\Proj T^*\un{M}$ so that $\D\subset TM$ corresponds to the standard contact distribution on $\Proj T^*\un{M}$.
Also, in this dimension,  $\Proj T^*\un{M}$ is canonically isomorphic to $\Proj T\un{M}$ and the LC structure on $M$ determines the so-called \emph{path structure} on $\un{M}$.
This is  a system of paths, i.e. unparametrized curves, on $\un{M}$ with the property that through each point in each direction there passes exactly one path from the system.
The relation is so that the paths on $\un{M}$ are just the projections of integral curves of the distribution $E$ (or, dually, of $F$).
In adapted local coordinates as above, the paths correspond to solutions of the 2nd-order ODE
\begin{align}\label{eq-ODE}
\ddot{y}=f(x,y,\dot{y}) ,
\end{align}
where $y=y(x)$ and $\dot{y}=\frac{dy}{dx}$,
cf. equation \eqref{eqn: LC PDE}.

The geometry of 2nd-order ODEs is a classical subject studied by Lie, Tresse, and others. 
There are two fundamental invariants, expressed in terms of the function $f$ and its partial derivatives, whose joint vanishing is equivalent to the triviality of the equation under point transformations.
It is well known that these two invariants correspond to the two harmonic curvature components of the associated LC structure.
The simpler one is just $f_{pppp}$, where the subscripts denote the partial derivatives with respect to the third variable of $f$.
The vanishing of this invariant means that the equation \eqref{eq-ODE} has the form 
\begin{align}\label{eq-ODE-proj}
\ddot{y}=A_0 +A_1 \dot{y} +A_2 \dot{y}^2 +A_3 \dot{y}^3 ,
\end{align}
i.e. the right-hand side is a cubic polynomial in $\dot{y}$, whose coefficients $A_i$ are functions of $x$ and $y$.
The previous restriction is equivalent to the fact that the equation is {geodesic}, i.e. its solutions are geodesics of an affine connection on $\un{M}$.
In terms of Section \ref{LC-by-projective}, this just means that the LC structure on $M$ is induced by a projective structure on $\un{M}$.
In such case, the path structure on $\un{M}$ is called projective.
A relation to the Christoffel symbols of a representative affine connection is as follows
\begin{align}
\label{ODE-coeff}
A_0=-\Gamma^2_{11}, \quad
A_1=\Gamma^1_{11}-2\Gamma^2_{12}, \quad
A_2=2\Gamma^1_{12}-\Gamma^2_{22}, \quad
A_3=\Gamma^1_{22} .
\end{align}
Most of these observations appear in Cartan's projective papers, cf. \cite{Cartan1924}, \cite{Cartan1937}.
The relation \eqref{ODE-coeff} is to be read as a special case of our general characterization in Proposition \ref{LC-proj-char}.

\subsection{Canonical curves} \label{curves-dim3}

For a general 3-dimensional LC structure on $M$, the harmonic curvature components have rather high homogeneity degree.
This allows us to express all components of the normal Cartan connection with much less effort than in general dimension; a full derivation can be found in \cite{Doubrov2016}.
With this equipment, many problems can be readily treated.
The typical instance related to our article concerns deducing systems of ODEs whose solutions are canonical curves in $M$ of given type.
We do not develop this approach here, see \cite{Doubrov2005} for a general strategy and examples.

Specializing to chains and null-chains, there is an alternative way which employs the Fefferman metric on the associated 4-dimensional Fefferman space $\wt{M}$.
This is due to the correspondence with null-geodesics in $\wt{M}$ as described in Theorem \ref{cor:integrable chains}.
With the notation as above, the coordinate expression of the metric from Theorem \ref{prop-Feff-integrable} is
\begin{align}
\wt{g} = \d x\odot(\d p-f\d x) + (\d y-p\d x)\odot \frac13 \Big( -\frac12 f_{pp}(\d y-p\d x) -2 f_p\d x + 2\d s\Big) .
\label{eq-feff-4}
\end{align}
This expression coincides with \cite[formula (31)]{Nurowski2003} and \cite[formula (25)]{Bor2022}, where it is derived by different means.
The system of ODEs for chains derived this way can be found in \cite[Proposition 3.6]{Bor2022}.

Specializing only to chains, we have yet another strategy employing the Kropina metric. 
More concretely, let $c:I\to M$ be a curve that is everywhere transverse to the contact distribution $\D\subset TM$.
Let the curve be parametrized by the first coordinate $x$, i.e. $c(x)=(x,y(x),p(x))$ where $\dot{y}(x)-p(x)\ne 0$, for all $x\in I$.
Let us consider the Kropina metric $F$ defined as in \eqref{eq-Kropina} corresponding to the local section of the Fefferman projection $\wt{M}\to M$ given by $s=0$.
This evaluated on the tangent vector field of $c$ gives
\begin{align}
F(\dot{c}) = (\dot{y}-p)^{-1} \Big( \dot{p}-f -\frac23 f_p (\dot{y}-p) -\frac16 f_{pp}(\dot{y}-p)^2 \Big) .
\end{align}
The Euler--Lagrange equations of this functional are
\begin{align}\label{eq-EL-p}
\frac{\del F}{\del p}-\frac{\d}{\d x}\frac{\del F}{\del \dot{p}} 
&=(\dot{y}-p)^{-2} \Big(\ddot{y} -f -f_p(\dot{y}-p) -\frac12 f_{pp}(\dot{y}-p)^2 -\frac16 f_{ppp}(\dot{y}-p)^3 \Big) =0, \\
\begin{split}
\frac{\del F}{\del y}-\frac{\d}{\d x}\frac{\del F}{\del \dot{y}} 
&= (\dot{y}-p)^{-3}\, \Big( -2(\ddot{y}-\dot{p})(\dot{p}-f) +(\ddot{p}-\dot{f})(\dot{y}-p) -f_y(\dot{y}-p)^2 + \\ 
& \hskip8em + \frac16f_{ppx}(\dot{y}-p)^3 -\frac23f_{py}(\dot{y}-p)^3 -\frac16f_{ppy}(\dot{y}-p)^4 \Big) =0 .
\label{eq-EL-y}
\end{split}
\end{align}
By Theorem \ref{prop-Kropina}, this is the system of ODEs whose solutions are precisely the chains.
Compared with other approaches, we consider the present one the most straightforward.

An interesting side issue of these observations is the following characterization of LC structures induced by projective, i.e. characterization of projective path structures.
Rewriting \eqref{eq-EL-p} as 
\begin{align}
\ddot{y}=f +f_p(\dot{y}-p) +\frac12 f_{pp}(\dot{y}-p)^2 +\frac16 f_{ppp}(\dot{y}-p)^3,
\label{eq-ODE-p}
\end{align}
the right-hand side can be interpreted as the third-order Taylor expansion of $f=f(x,y,p)$ at $p=\dot{y}$, for $x$ and $y$ fixed. 
Thus, the right-hand side equals to $f$ if and only if $f$ is a cubic polynomial in $p$, i.e. the projected equation has the form \eqref{eq-ODE-proj}. 
This leads to the following criterion, originally proved in \cite{Bor2022}:

\begin{prop}[{\cite[Theorem 1]{Bor2022}}]
\label{prop: Bor-Wilse}
Let the LC structure on $M=\Proj T\un{M}$ correspond to a 2-dimensional path structure on $\un{M}$.
Then the path structure is projective if and only if all chains in $M$ project to the paths in $\un{M}$.
\end{prop}

On the one hand, for a fixed initial condition $(x_0,y_0,\dot{y}_0)$ on $\un{M}$, there is a unique path, say $\un{c}$, obeying this condition.
On the other hand, this initial condition lifts to a 2-parameter family of initial conditions $(x_0,y_0,p_0,\dot{y}_0,\dot{p}_0)$ on $M$, each of which determines a unique chain provided that $\dot{y}_0-p_0\ne 0$, i.e. the initial tangent vector is transverse to $\D\subset TM$.
This family of chains projects to a family of curves in $\un{M}$ obeying the initial condition for $\un{c}$. 
Generally, the behavior is such that the curves in the family are distinct and only converge to the path $\un{c}$ as $p_0\to\dot{y}_0$, cf. \eqref{eq-ODE-p}.
It is the special feature of projective path structures that all these curves coincide.

Note that the limit curve in $M$ corresponding to $p_0=\dot{y}_0$ is just an integral curve of the distribution $E$.
In our previous terminology, this is just a null-chain of type $E$.

\subsection{Canonical families of paths}
Since $M=\Proj T\un{M}$, any curve $c:I\to M$ is a 1-parameter family of tangent directions along its point-projection $\un{c}:I\to\un{M}$.
Interpreting each direction as an initial condition, we get a 1-parameter family of paths in $\un{M}$.
The curve $c$ is tangent, respectively transverse, to the contact distribution $\D\subset TM$ if and only if the projected curve $\un{c}$ is tangent, respectively transverse, to the paths just described.

Canonical curves in $M$ determine canonical 1-parameter families of paths in $\un{M}$.
Such families provide a complex (and visually appealing) picture of the path structure.
Focusing on chains in the homogeneous model, we know from claim \eqref{chain-b} of Proposition \ref{prop:model-chain} that all paths from the corresponding family have a common intersection point, cf. \cite[figure 2]{Bor2022}.
It is not \emph{a priori} clear to which extent this property generalizes.
Experiments with non-trivial examples expose various kinds of behavior.
In the following example we show that the property does not hold automatically even for projective path structures.

\begin{example} \label{ex}
Let the LC structure be given by the defining function 
$f(x,y,p)=\frac12(p+e^{-2x}p^3)$, 
i.e. the corresponding ODE \eqref{eq-ODE} has the form 
\begin{align}
\ddot{y} = \frac{1}{2} \left( \dot{y}+e^{-2x}\dot{y}^3 \right).
\label{eq-ODE-ex}
\end{align}
This determines a projective path structure with relatively many symmetries, cf. \cite[Section 2.6]{bryantsolution}.
The general, respectively the singular, solutions to the equation \eqref{eq-ODE-ex} are
\begin{align*}
y(x)= \pm 2\, C_1^{-1} \sqrt{C_1 e^x+1}+C_2, \quad
y(x)=\pm e^x+C_2, \quad y(x)=C_2,
\end{align*}
where $C_1\ne 0$ and $C_2$ are arbitrary constants. 
In particular, the constant function $y(x)=0$ is a solution, i.e. the $x$-coordinate axis is a path.
We will analyze a particular chain projecting to this path.

Substituting the current setting into the Euler--Lagrange system for chains, the equation \eqref{eq-EL-p} is satisfied automatically (since our path structure is projective) and the equation \eqref{eq-EL-y} reduces to 
\begin{align*}
\ddot{p}p-2\dot{p}^2+\frac12\dot{p}p=0.
\end{align*}
With the transversality condition, the general solution to this equation is
\begin{align}
p(x)=\frac{\sqrt{e^x}}{D_1 \sqrt{e^x}+D_2},
\label{eq-p}
\end{align}
where $D_1$ and $D_2$ are arbitrary constants. 
Particular solutions to \eqref{eq-ODE-ex} that are transverse to the path $y(x)=0$ are
\begin{align}
\gamma_0(x)=e^x-1,\quad \gamma_1(x)=2e\sqrt{-e^{x-1}+1}.
\label{eq-gam}
\end{align}
The chain determined by the corresponding boundary condition is given by \eqref{eq-p} with
$D_1=\frac{-1}{\sqrt{e}-1}$ and $D_2=\frac{\sqrt{e}}{\sqrt{e}-1}$,
i.e. by
\begin{align}
p(x) = \frac{\sqrt{e^{x+1}}-\sqrt{e^x}}{\sqrt{e}-\sqrt{e^x}}.
\label{eq-ch}
\end{align}
Now, the paths \eqref{eq-gam} have a unique intersection, namely, 
\begin{align}
x=\ln\left(2\sqrt{e(2e-1)}-(2e-1)\right), \quad
y=2\sqrt{e(2e-1)}-2e.
\label{eq-int}
\end{align}
It is a matter of direct computation that the path determined by an arbitrarily chosen initial condition \eqref{eq-ch} along $y(x)=0$ does not pass through the point \eqref{eq-int}.
\end{example}


\end{document}